\documentclass[11pt]{article}

\usepackage{verbatim}
\usepackage{graphicx}
\usepackage{amsmath,amssymb,amsfonts,euscript,enumerate,latexsym}
\usepackage{amsthm}
\usepackage{color,wrapfig}
\usepackage{pxfonts}
\usepackage{fancyhdr}
\usepackage{mathrsfs}
\usepackage{mathtools}
\usepackage{epsfig,subfigure}
\usepackage{marvosym}
\usepackage{txfonts}
\usepackage{hyperref}
\usepackage{times}
\usepackage{enumerate}

\def\titlerunning#1{\gdef\titrun{#1}}
\makeatletter
\def\author#1{\gdef\autrun{\def\and{\unskip, }#1}\gdef\@author{#1}}
\def\address#1{{\def\and{\\\hspace*{18pt}}\renewcommand{\thefootnote}{}%
\footnote {#1}}%
\markboth{\autrun}{\titrun}}
\makeatother
\def\email#1{e-mail: #1}

\def\keywords#1{\par\medskip
\noindent\textbf{Keywords.} #1}

\usepackage{etoolbox}
\makeatletter
\patchcmd{\@maketitle}
  {\ifx\@empty\@dedicatory}
  {\ifx\@empty\@date \else {\vskip3ex \centering\footnotesize\@date\par\vskip1ex}\fi
   \ifx\@empty\@dedicatory}
  {}{}
\patchcmd{\@adminfootnotes}
  {\ifx\@empty\@date\else \@footnotetext{\@setdate}\fi}
  {}{}{}
\makeatother

\newtheorem{thm}{Theorem}[section]
\newtheorem{cor}[thm]{Corollary}
\newtheorem{lemma}[thm]{Lemma}

\newtheorem{remark}[thm]{Remark}

\newcommand{\R}{{\mathbb{R}}}

\newcommand{\N}{{\mathbb{N}}}

\newcommand{\vp}{\varphi}
\newcommand{\osc}{\operatornamewithlimits{osc}}

\newcommand{\D}{\nabla}

\newcommand{\La}{\triangle}
\newcommand{\bs}{\backslash}

\begin{document}

\baselineskip=17pt

\titlerunning{}

\title{Local Continuity and Asymptotic Behaviour of  Degenerate Parabolic Systems}

\author{ Sunghoon Kim 
\and Ki-Ahm Lee }

\date{}

\maketitle

\address{ 
S. Kim (\Letter) :
Department of Mathematics, School of Natural Sciences, The Catholic University of Korea,\\
43 Jibong-ro, Wonmi-gu, Bucheon-si, Gyeonggi-do, 14662, Republic of Korea ;\\
\email{math.s.kim@catholic.ac.kr}
\and
Ki-Ahm Lee :
Department of Mathematical Sciences, Seoul National University, Gwanak-ro 1, Gwanak-Gu, Seoul, 08826, South Korea \&
Korea Institute for Advanced Study, Seoul 02455, Korea;\\
\email{ kiahm@snu.ac.kr }
}

\begin{abstract}
We study the local H\"older continuity and the asymptotic behaviour of solution, $\bold{u}=(u^1,\cdots, u^k)$, of the degenerate system
\begin{equation*}
u^i_t=\nabla\cdot\left(m\,U^{m-1}\nabla u^i\right) \qquad \text{for $m>1$ and $i=1,\cdots,k$ }
\end{equation*}
which describes the population densities of $k$-species whose diffusions are determined by their total population density $U=u^1+\cdots+u^k$. For the local H\"older continuity, we adopt the intrinsic scaling and iteration arguments of DeGiorgi, Moser, and Dibenedetto. Under some regularity conditions, we also prove that the population density of $i$-th species with the population $M_i$ converges in $C^{\infty}_s$ to the function $\frac{M_i}{M}\mathcal{B}_M(x,t)$ as $t\to \infty$ where $\mathcal{B}_M$ is the Barenblatt profile of the standard porous medium equation with $L^1$ mass $M=M_1+\cdots+M_k$. As a consequence of asymptotic behaviour, it is shown that each density becomes a concave function after a finite time.
\end{abstract}

\setcounter{equation}{0}
\setcounter{section}{0}

\section{Introduction}\label{section-intro}
\setcounter{equation}{0}
\setcounter{thm}{0}
For a given number $k\in\N$ , let $u^i\geq 0$, $\left(i=1,\cdots,k\right)$, represent the population density of $i$-th species and  $U$ be the total density of all species, i.e., 
\begin{equation}\label{def-total-population-U}
U=u^1+u^2+\cdots+u^{k}=\sum_{i=1}^{k}u^i.
\end{equation}
As a simplest case, we consider a system whose diffusion coefficients are controlled by the total population density $U$, i.e., let $\bold{u}=\left(u^1,\cdots, u^k\right)$ be a solution of 
\begin{equation}\label{eq-pme_system}\tag{SPME}
u^i_t=\nabla\cdot\left(m\,U^{m-1}\nabla u^i\right) \qquad \text{for $m>1$ and $i=1,\cdots,k$ }.
\end{equation}
By \eqref{def-total-population-U}, $U$ satisfies the standard porous medium equation (shortly, PME)
\begin{equation*}\label{eq-PME-satisfied-by-U-total-species}
U_t=\sum_{i=1}^k\left(u^i\right)_t=\sum_{i=1}^k\nabla\cdot\left(m\,U^{m-1}\nabla u^i\right)=\nabla\cdot\left(m\,U^{m-1}\nabla U\right)=\La U^m  \qquad \forall (x,t)\in\R^n\times(0,\infty)
\end{equation*} 
and
\begin{equation}\label{eq-natural-condition-between-density-of-one-species-and-total-density}
u^i(x,t)\leq U(x,t) \qquad \forall (x,t)\in\R^n\times\left[0,\infty\right),\,\,i=1,\cdots,k.
\end{equation}
Moreover, by the regularity theory for the standard PME it is well known that the function $U$ is locally bounded by $\left\|U\left(x,0\right)\right\|_{L^{1}\left(\R^n\right)}$ in $\R^n\times\left(0,\infty\right)$ (See Lemma \ref{lem-cf-chapter-9-of-Va1} for a detail), i.e., there exists a constant $C>0$ such that
\begin{equation*}
\left|u^i(x,t)\right|\leq \left|U\left(x,t\right)\right|\leq C\frac{\left\|U\left(x,0\right)\right\|_{L^1\left(\R^n\right)}}{t^{\frac{n}{n(m-1)+2}}}, \qquad \forall (x,t)\in\R^n\times\left(0,\infty\right).
\end{equation*}
Since both $u^i$ and $U$ satisfy the same equation, the relation \eqref{eq-natural-condition-between-density-of-one-species-and-total-density} can be obtained by their initial conditions, i.e.,
\begin{equation*}
u^i(x,0)\leq U(x,0) \qquad \forall x\in\R^n,\,\,i=1,\cdots,k.
\end{equation*}
\indent Let $0\leq U_0\in L^1\left(\R^n\right)\cap L^{1+m}\left(\R^n\right)$ be a compactly supported function. As a general case which covers above situation, we are going to study the local continuity and asymptotic behaviour of the problem
\begin{equation}\label{eq-for-each-population-u-i}\tag{$PME_u$}
\begin{cases}
\begin{aligned}
u_t&=\nabla\cdot\left(m\,U^{m-1}\nabla u\right)\qquad \qquad\mbox{in $\R^n\times(0,\infty)$}\\
u(x,0)&=u_0(x) \qquad \qquad \qquad \qquad\forall x\in\R^n
\end{aligned}
\end{cases}
\end{equation}
in the range of exponents $m>1$, with initial data $u_0$ satisfying
\begin{equation}\label{eq-basic-property-of-u-smaller-than-U-at-zero}
0\leq u_0(x)\leq U_0(x) \qquad \forall x\in\R^n
\end{equation}
where the diffusion coefficients $U$ is the solution of 
\begin{equation}\label{eq-PME-satisfied-by-U-total-species}\tag{$PME$}
\begin{cases}
\begin{aligned}
U_t&=\nabla\cdot\left(m\,U^{m-1}\nabla U\right)=\La U^m \qquad\mbox{in $\R^n\times(0,\infty)$}\\
U(x,0)&=U_0(x) \qquad \qquad \qquad \qquad \qquad\forall x\in\R^n.
\end{aligned}
\end{cases}
\end{equation}
\indent If $U$ is equivalent to the solution $u$ in the sense that
\begin{equation*}
U(x,t)=cu^{\beta}(x,t) \qquad \forall(x,t)\in\R^n\times\left[0,\infty\right)
\end{equation*}
for some constants $c>0$ and $\beta\in\R^+$, then the equation appears in many physical phenomenons \cite{Ar}, \cite{DK}, \cite{Va1}. When $\beta (m-1)+1>1$, it is well known as the porous medium equation which arises in describing the flow of an ideal gas through a homogeneous porous medium \cite{Ar}. Since $\beta(m-1)>0$, the porous medium equation becomes degenerate when $u=0$ and this degeneracy let the flow propagate slowly with finite speed. This implies that there exists an interface or free boundary which separates regions where $u>0$ from regions where $u=0$, \cite{Va1}.  When $\beta (m-1)+1=1$ and $\beta (m-1)+1<1$, we call them the heat equation and the fast diffusion equation, respectively. Similar to the porous medium equation, the fast diffusion equation arises in many famous flows such as Yamabe flow and Ricci flow. we refer the readers to the papers \cite{PS} for Yamabe flow and to the papers \cite{Wu} for Ricci flow.\\
\indent There are many studies on the regularity and asymptotic behaviour for the porous medium and fast diffusion equations. We refer the readers to the papers \cite{CF1}, \cite{CF2}, \cite{CF3}, \cite{CW}, \cite{Di}, \cite{HU}, \cite{KL3}, \cite{LV}  for regularity and to the papers \cite{BBDGV}, \cite{DH},  \cite{HK1}, \cite{HK2}, \cite{HKs}, \cite{Va2} for asymptotic behaviours of solutions of porous medium and fast diffusion equations. \\
\indent Corresponding to the porous medium type equation, we can also derive the $p$-laplacian equation from \eqref{eq-for-each-population-u-i} by considering the diffusion coefficients $U^{m-1}$ as the gradients of the solution, i.e.,
\begin{equation*}
U^{m-1}=c\left|\nabla u\right|^{p-2} \qquad \mbox{in $\R^n\times\left[0,\infty\right)$}
\end{equation*}
for some constant $c>0$ and $p>1$. Large number of literatures on the local continuity and asymptotic behaviour of solutions of $p$-laplacian equation can be also found. We refer the readers to the papers \cite{CD}, \cite{DF} for various estimates about local continuity and to the paper \cite{KV} for the asymptotic behaviours of solution of $p$-laplacian equation.\\
\indent As the first result of this paper, we will investigate the local continuity of the parabolic partial differential equation 
\begin{equation}\label{eq-standard-form-for-local-continuity-estimates-intro-1}
u_t=\nabla\cdot\left(mU^{m-1}\mathcal{A}\left(\nabla u,u,x,t\right)+\mathcal{B}\left(u,x,t\right)
\right), \qquad \left(m>1\right)
\end{equation}
which is a generalized version of \eqref{eq-for-each-population-u-i}. The measurable functions $\mathcal{A}$ and $\mathcal{B}$ are assumed to satisfy the following conditions 
\begin{equation}\label{eq-condition-1-for-measurable-functions-mathcal-A-and-mathcal-B}
\mathcal{A}\left(\nabla u,u,x,t\right)\cdot\nabla u\geq C_1\left|\nabla u\right|^2-u^2f_1
\end{equation}
\begin{equation}\label{eq-condition-2-for-measurable-functions-mathcal-A-and-mathcal-B}
\left|\mathcal{A}\left(\nabla u,u,x,t\right)\right|\leq C_2\left|\nabla u\right|+uf_2
\end{equation}
\begin{equation}\label{eq-condition-3-for-measurable-functions-mathcal-A-and-mathcal-B}
\left|\mathcal{B}\left(u,x,t\right)\right|\leq u^bf_3
\end{equation}
for positive constants $b\geq 1$, $C_i>0$ and functions $0\leq f_i\in L^2\cap L^{\infty}$, $\left(i=1,2,3\right)$. Among the methods for the local continuity, we will take the oscillation argument which will be used often for the H\"older regularity of solutions. The main step of the oscillation argument is to prove that the difference between supremum and infimum on a chosen set decreases proportionally as the radius of the set shrinks to half (Oscillation Lemma). To make our methods work, some conditions are needed to be imposed on $U$, which cover \eqref{eq-natural-condition-between-density-of-one-species-and-total-density}, to deal with the difficulties stem from the degeneracy. In that point of view, we generally assume that the function $U$ satisfies the following assumption:
\begin{description}
\item[\textbf{Assumption I :}]There exist uniform constants $\lambda>0$ and $\beta\geq 0$ such that
\begin{equation}\label{eq-assumption-through-the-regularity-section-0}
\begin{aligned}
\lambda u^{\,\beta}\leq U\qquad \mbox{locally in $\R^n\times\left(0,\infty\right)$}.
\end{aligned}
\end{equation}
\end{description}
If some conditions are imposed on the measurable functions $\mathcal{A}$ and $\mathcal{B}$, and if the function $U$ is a subsolution of a parabolic partial differential equation which contains $\mathcal{A}$ and $\mathcal{B}$ in it, then $U$ may  be locally bounded in $\R^n\times\left(0,\infty\right)$. (An example of local boundedness of $U$ is presented in Lemma \ref{eq-example-boundedness-of-diffusion-coefficients-of-generaized-equation}). Hence we can also assume the following condition:
\begin{description}
\item[\textbf{Assumption II :}]There exists an uniform constant $\Lambda<\infty$  such that
\begin{equation}\label{eq-assumption-through-the-regularity-section-0}
\begin{aligned}
U\leq \Lambda\qquad \mbox{locally in $\R^n\times\left(0,\infty\right)$}.
\end{aligned}
\end{equation}
\end{description}

With this assumption, we now state the first result of our paper.
\begin{thm}\label{eq-local-continuity-of-solution}
Under the \textbf{Assumption I and II}, any solution of \eqref{eq-standard-form-for-local-continuity-estimates-intro-1} is locally continuous in $\R^n\times\left(0,\infty\right)$.
\end{thm}
\begin{remark}\label{eq-remark-for-holder-estimates-for-function-u-i-s-in-solution-bold-u}
 If $U$ is equivalent to $u^{\beta}$ in $\R^n\times(0,\infty)$, i.e., there exist some constants $0<c\leq C<\infty$ such that
\begin{equation*}
cu^{\beta}\leq U\leq Cu^{\beta} \qquad \mbox{in $\R^n\times(0,\infty)$},
\end{equation*}
then we can find the modulus of continuity of solution of \eqref{eq-standard-form-for-local-continuity-estimates-intro-1}, i.e., the solution $u$ is locally H\"older continuous in $\R^n\times\left(0,\infty\right)$.
\end{remark}
\indent As a consequence of Remark \ref{eq-remark-for-holder-estimates-for-function-u-i-s-in-solution-bold-u}, we can have the following local estimates of solution $\bold{u}=\left(u^1,\cdots,u^k\right)$ of \eqref{eq-pme_system}.
\begin{thm}\label{cor-local-holder-estimates-of-u-i-s-in-solution-bold-u}
Let $m>1$ and let $\bold{u}=\left(u^1,\cdots,u^k\right)$ be solution of \eqref{eq-pme_system}  with initial data $u^i_0\in L^1\left(\R^n\right)\cap L^{1+m}\left(\R^n\right)$ nonnegative and compactly supported for all $1\leq i\leq k$. Then the function $u^i$, $\left(1\leq i\leq k\right)$, is locally H\"older continuous in $\R^n\times(0,\infty)$. Especially, all components of solution $\bold{u}$ have the same modulus of continuity.  
\end{thm}
\begin{remark}
The local continuity of Theorem \ref{eq-local-continuity-of-solution} and Theorem \ref{cor-local-holder-estimates-of-u-i-s-in-solution-bold-u} can be extended to the fast diffusion type system, i.e., Theorem \ref{eq-local-continuity-of-solution} and Theorem \ref{cor-local-holder-estimates-of-u-i-s-in-solution-bold-u} also holds for the range of exponents $0<m<1$.
\end{remark}
\indent As the second result of this paper, we will deal with the asymptotic behaviour of \eqref{eq-for-each-population-u-i}. Denote by $\mathcal{B}_M$ the self-similar Barenblatt solution of the porous medium equation with $L^1$ mass $M>0$. If the function $U_0$ has the mass $M$ in $L^1\left(\R^n\right)$, then by \cite{LV} it is well known that
\begin{equation*}
U(\cdot,t)\to \mathcal{B}_M(\cdot,t) \qquad \mbox{in $C^{\infty}$} \quad \mbox{as $t\to\infty$}
\end{equation*}
under some regularity conditions of $U$. If there is a limit of the solution $u$ of \eqref{eq-for-each-population-u-i}, then the limit $u_{\infty}$ will satisfy
\begin{equation}\label{eq-expectation-equation-for-limit-solution-as-t-to-infty}
\left(u_{\infty}\right)_t=\nabla\cdot\left(m\,\mathcal{B}_{M}^{m-1}\nabla u_{\infty}\right)\qquad \mbox{and} \qquad u_{\infty}\leq \mathcal{B}_{M} \qquad \forall (x,t)\in\R^n\times[0,\infty). 
\end{equation}
Since $c\mathcal{B}_M$ is also a solution of \eqref{eq-expectation-equation-for-limit-solution-as-t-to-infty} for any constant $c\in\R^+$, it could be strongly expected that
\begin{equation*}
u_{\infty}=\frac{\left\|u_0\right\|_{L^1}}{M}\mathcal{B}_M
\end{equation*}
if the solution $u$ maintains its $L^1$-mass. Under this expectation, we are going to state our second result of paper.
\begin{thm}\label{eq-convergence-in-L-1-and-L-infty-of-u-to-some-constant-Barenblatt}
Let $U$ be a solution of \eqref{eq-PME-satisfied-by-U-total-species}  with initial data $U_0$ nonnegative, integrable and compactly supported. Suppose that $u$ is a solution of \eqref{eq-for-each-population-u-i} satisfying \eqref{eq-basic-property-of-u-smaller-than-U-at-zero}. Then
\begin{equation}\label{eq-convergence-in-L-1-from-u-to-barrenblatt-M-0-over-M-B-M}
\lim_{t\to\infty}\left\|u\left(\cdot,t\right)-\frac{M_0}{M}\mathcal{B}_{M}\left(\cdot,t\right)\right\|_{L^1}=0
\end{equation}
and
\begin{equation}\label{eq-convergence-in-L-infty-from-u-to-barrenblatt-M-0-over-M-B-M}
\lim_{t\to\infty}t^{a_1}\left|u(x,t)-\frac{M_0}{M}\mathcal{B}_M(x,t)\right|=0 \qquad \mbox{uniformly in $\R^n$}
\end{equation}
for some constant $a_1$ where  $\left\|U_0\right\|_{L^1\left(\R^n\right)}=M$ and $\left\|u_0\right\|_{L^1\left(\R^n\right)}=M_0$.
\end{thm}
By Theorem \ref{eq-convergence-in-L-1-and-L-infty-of-u-to-some-constant-Barenblatt}, we can get $L^{1}$ and $L^{\infty}$ convergence of solution $\bold{u}=\left(u^1,\cdots,u^k\right)$ of \eqref{eq-pme_system}.
\begin{cor}\label{cor-convergence-in-L-1-and-L-infty-of-u-to-some-constant-Barenblatt}
Let $m>1$. For each $1\leq i\leq k$, let $u_0^i(x)$ be nonnegative, integrable and compactly supported with 
\begin{equation*}
\left\|u_0^i\right\|_{L^1\left(\R^n\right)}=M_i>0.
\end{equation*} 
Suppose that $\bold{u}=\left(u^1,\cdots,u^k\right)$ is a solution of \eqref{eq-pme_system} with initial data $u^i_0$. Then
\begin{equation}\label{eq-convergence-in-L-1-from-u-to-barrenblatt-M-0-over-M-B-M}
\lim_{t\to\infty}\left\|u^i\left(\cdot,t\right)-\frac{M_i}{M}\mathcal{B}_{M}\left(\cdot,t\right)\right\|_{L^1}=0
\end{equation}
and
\begin{equation}\label{eq-convergence-in-L-infty-from-u-to-barrenblatt-M-0-over-M-B-M}
\lim_{t\to\infty}t^{a_1}\left|u(x,t)-\frac{M_i}{M}\mathcal{B}_M(x,t)\right|=0 \qquad \mbox{uniformly in $\R^n$}
\end{equation}
for some constant $a_1$ where  $M=M_1+\cdots+M_k$.
\end{cor}
\indent Denote by $v$ the pressure of $u$, i.e., 
\begin{equation*}
v(x,t)=m\,u^{m-1}(x,t) \qquad \forall (x,t)\in\R^n\times\left[0,\infty\right).
\end{equation*}
For any $\lambda>0$, let $v_{\lambda}$ be the rescaled function of $v$ by
\begin{equation*}
v_{\lambda}(x,t)=\lambda^{\frac{(m-1)n}{(m-1)n+2}}v\left(\lambda^{\frac{1}{(m-1)n+2}}x,\lambda t\right), \qquad \forall \lambda>0,\,\,(x,t)\in\R^n\times(0,\infty).
\end{equation*}
By Theorem \ref{eq-convergence-in-L-1-and-L-infty-of-u-to-some-constant-Barenblatt}, there is the uniform convergence such that
\begin{equation*}
v_{\lambda}(x,t)\to \left(\frac{M_0}{M}\mathcal{B}_M(x,t)\right)^{m-1} \qquad  \mbox{in $L^p,\,\left(p\geq 1\right)$}\quad \mbox{as $\lambda\to\infty$}.
\end{equation*}
By $C^{\infty}$ regularity in \cite{Ko} and an argument similar to the proof of Theorem 3.2 of \cite{LV}, we can extend our convergence in $L^p$, $\left(p\geq 1\right)$, to the one in $C_s^{\infty}$ where $ds$ is a Riemannian metric which will be mentioned later.  For the $C^{\infty}_s$ convergence, following conditions are also needed to be imposed on the initial data $u_0$ and $U_0$, see  \cite{CVW} for the detail.\\
\\
\textbf{Conditions for $C_s^{\infty}$-convergence}
\begin{itemize}
\item Support : $\textbf{supp}\,u_0\,=\,\textbf{supp}\,U_0$.
\item Regularity : $u^{m-1}_0$, $U^{m-1}_0\in C^1\left(\overline{\Omega_0}\right)$.
\item Non-degeneracy : there exists a constant $K>0$ such that 
\begin{equation}\label{eq-non-degeneracy-of-pressure-of-sol-u}
0<\frac{1}{K}<u^{m-1}_0+\left|\nabla u^{m-1}_0\right|<K \qquad \mbox{and} \qquad 0<\frac{1}{K}<U^{m-1}_0+\left|\nabla U^{m-1}_0\right|<K\qquad \mbox{in $\overline{\Omega_0}$}.
\end{equation}
\end{itemize}
where $\Omega_0$ is the set of all points in $\R^n$ where $U_0>0$, i.e.,
\begin{equation*}
\Omega_0=\left\{x\in\R^n:U_0(x)>0\right\}.
\end{equation*} 
Under the \textbf{Conditions for $C_s^{\infty}$-convergence}, the $C_s^{\infty}$ convergence of pressure $v$ is stated as follow.
\begin{thm}[cf. Theorem 3.2 of \cite{LV}]\label{thm-Theorem-3-2-of-cite-LV}
Under the assumption of Theorem \ref{eq-convergence-in-L-1-and-L-infty-of-u-to-some-constant-Barenblatt} and \textbf{Conditions for $C^{\infty}_s$-convergence}, the rescaled function $v_{\lambda}$ satisfies
\begin{equation*}
v_{\lambda}(x,1)\to  \left(\frac{M_0}{M}\mathcal{B}_M(x,1)\right)^{m-1}\qquad \mbox{in $C_s^{k}$} \quad \mbox{ as $\lambda\to\infty$}
\end{equation*}
for any $k\in\N$.
\end{thm}
As a consequence of Theorem \ref{thm-Theorem-3-2-of-cite-LV}, we can also get the following geometric properties of pressure $v$.
\begin{cor}[cf. Theorem 3.3 of \cite{LV}]
There exists a constant $t_0>0$ such that the pressure $v(x,t)$ is strictly concave on $\left\{x\in\R^n:v(x,t)>0\right\}$ for all $t>t_0$. More precisely
\begin{equation*}
\lim_{t\to\infty}\,t\,\frac{\partial^2\,v}{\partial\,x_i^2}=-\frac{1}{(m-1)n+2} \qquad \mbox{uniformly in $x\in \textbf{supp}\,v$}\qquad  \left(\forall i=1,\,\cdots,\,n\right).
\end{equation*}
\end{cor}
As a consequence of the Theorem \ref{eq-convergence-in-L-1-and-L-infty-of-u-to-some-constant-Barenblatt} and Theorem \ref{thm-Theorem-3-2-of-cite-LV}, we can describe the large time asymptotic behaviours of solutions of \eqref{eq-pme_system} as $t\to\infty$. 
\begin{cor}
For each $1\leq i\leq k$, let $u_0^i(x)$ be nonnegative, integrable and compactly supported with 
\begin{equation*}
\left\|u_0^i\right\|_{L^1\left(\R^n\right)}=M_i>0.
\end{equation*} 
Suppose that $\bold{u}=\left(u^1,\cdots,u^k\right)$ be a solution of \eqref{eq-pme_system}. 
Then, under the Non-degeneracy of $u_0^i$, $(i=1,\cdots,k)$, the pressure 
\begin{equation*}
v^i(x,t)=m\left(u^i(x,t)\right)^{m-1} \qquad \forall (x,t)\in\R^n\times[0,\infty),\,\,1\leq i\leq k.
\end{equation*} 
 convergence to $\left(\frac{M_i}{M}\mathcal{B}_{M}\right)^{m-1}$ uniformly in $L^p$, $\left(p\geq 1\right)$ and $C_s^{\infty}$ as $t\to\infty$ where $M=M_1+\cdots+M_k$.\\
\indent As a consequence of $C_s^{\infty}$ convergence, the pressure of $v^i$ becomes strictly concave on $\left\{x\in\R^n: v>0\right\}$ after a finite time.
\end{cor}
Many studies on the degenerate system  can be found. We refer the readers to the paper \cite{KMV} for the system of degenerate parabolic equations idealizing reactive solute transport in porous media. In \cite{KMV}, they show the existence of a unique weak solution to the coupled system and derive regularity estimates.\\
\indent We end up this section by introducing the definition of solutions. We say that $u$ is a weak solution of \eqref{eq-for-each-population-u-i} in $\R^n\times(0,T)$ if $u$ is a locally integrable function satisfying
\begin{enumerate}
\item $u$ belongs to function space:
\begin{equation*}\label{eq-first-condition-of-weak-soluiton-u-with-U}
U^{m-1}\left|\nabla u\right|\in L^2\left(0,T:L^2\left(\R^n\right)\right).  
\end{equation*}
\item $u$ satisfies the identity: 
\begin{equation}\label{eq-identity--of-formula-for-weak-solution}
\int_{0}^{T}\int_{\R^n}\left\{m\,U^{m-1}\nabla u\cdot\nabla\vp-u\vp_t\right\}\,dxdt=\int_{\R^n}u_0(x)\vp(x,0)\,dx
\end{equation}
holds for any test function $\vp\in C^{1}\left(\R^n\times(0,T)\right)$ which has a compact support in $\R^n$ and vanishes for $t=T$.
\end{enumerate}

\indent This paper is divided into three parts: In Part 1 (Section 2) we study the properties of the solution of \eqref{eq-for-each-population-u-i}. Part 2 (Section 3) is devoted to the proof of local continuity of solution of \eqref{eq-for-each-population-u-i} and local H\"older continuity of solution of \eqref{eq-pme_system}, (Theorem \ref{eq-local-continuity-of-solution} and Theorem \ref{cor-local-holder-estimates-of-u-i-s-in-solution-bold-u}). As mentioned above, the main step is to show the Oscillation Lemma. In Part 3 (Section 4), we will investigate the $C^{\infty}_s$ convergence between the solution and Barenblatt solution under some regularity conditions and degeneracy of equation.

\section{Preliminary Results}
\setcounter{equation}{0}
\setcounter{thm}{0}

In this section, we will study the existence and properties of solutions $u$ and $U$ of \eqref{eq-for-each-population-u-i} and \eqref{eq-PME-satisfied-by-U-total-species}, respectively. 

\subsection{Properties of solution $U$ of the porous medium equation}

As the first step of this section, we are going to deal with the existence and properties of function $U$ which appears in the diffusion coefficients of the system \eqref{eq-pme_system}. The first one is the {\it existence of weak solution} and the next one is the {\it mass conservation} of \eqref{eq-PME-satisfied-by-U-total-species}.

\begin{lemma}[cf. Chapter 9 of \cite{Va1}]\label{lem-cf-chapter-9-of-Va1}
Let $m>1$. For every $U_0\in L^1\left(\R^n\right)\cap L^{m+1}\left(\R^n\right)$ there exists an unique weak solution $U$ of \eqref{eq-PME-satisfied-by-U-total-species} with initial data $U_0$ such that $U^m\in L^2\left(0,\infty:H^1\left(\R^n\right)\right)$. The solution $U$ satisfies estimates
\begin{equation}\label{eq-L-infty-bound-of-diffusion-coefficients-U-by-L-infty-of-U-0-with-positive-t}
\left|U(x,t)\right|\leq C\left\|U_0\right\|_{1}^{2a_2}t^{-a_1}
\end{equation}
where $a_1=\frac{n}{n\left(m-1\right)+2}$, $a_2=\frac{1}{n(m-1)+2}$ and $C>0$ depends only on $m$ and $n$. If $U_0\in L^p\left(\R^n\right)$ for $1\leq p\leq\infty$, then $U\left(\cdot,t\right)\in L^p\left(\R^n\right)$ and 
\begin{equation*}
\left\|U\left(\cdot,t\right)\right\|_{L^p}\leq \left\|U_0\right\|_{L^p}.
\end{equation*}
\end{lemma}

\begin{lemma}[Mass conservation of PME in \cite{Va1}]\label{lem-Mass-conservation-of-PME}
Under the hypothesis of Lemma \ref{lem-cf-chapter-9-of-Va1}, we have
\begin{equation*}
\int_{\R^n}U(x,t)\,dx=\int_{\R^n}U_0(x)\,dx \qquad \mbox{for every $t>0$}.
\end{equation*}
\end{lemma}

\subsection{Uniqueness and existence of solution $u$ of \eqref{eq-for-each-population-u-i}}

With the properties of $U$, we will consider the uniqueness and existence of weak solution $u$ of \eqref{eq-for-each-population-u-i}.

\begin{lemma}[Uniqueness of solutions]\label{lem-uniquness-of-weak-solution-2}
The Problem \eqref{eq-for-each-population-u-i} has at most one weak solution if $u\in L^2\left(\R^n\right)$.
\end{lemma}

\begin{proof}
Let $u_1$ and $u_2$ be two solutions of \eqref{eq-for-each-population-u-i} with initial data $u_{0,1}$ and $u_{0,2}$ respectively. Then $v=u_1-u_2$ is also a solution of \eqref{eq-for-each-population-u-i} with initial data $v_0=u_{0,1}-u_{0,2}$. Then we have
\begin{equation}\label{eq-energy-type-inequality-0of-=difference-u-1-and-u-2}
\int_{0}^T\int_{\R^n}m\,U^{m-1}\left|\nabla v_+\right|^2\,dxdt+\frac{1}{2}\int_{\R^n}v_+^2\left(x,T\right)\,dx\,\leq\,\frac{1}{2}\int_{\R^n}v_+^2\left(x,0\right)\,dx.
\end{equation}
Thus if  $u_{0,1}(x)\leq u_{0,2}(x)$ for all $x\in\R^n$, i.e., $\left(v_0\right)_+(x)=0$ for all $x\in\R^n$, then by \eqref{eq-energy-type-inequality-0of-=difference-u-1-and-u-2}
\begin{align}
&v_+\left(x,t\right)=0 \qquad \mbox{a.e. in $\R^n\times(0,T)$}\notag\\
&\Rightarrow \qquad u_1(x,t)\leq u_{2}(x,t) \qquad \mbox{a.e. in $\R^n\times(0,T)$}. \label{eq-half-inequality-for-uniqueness-of-solution-of-main-1}
\end{align}
Similarly, we can also have
\begin{equation}\label{eq-half-inequality-for-uniqueness-of-solution-of-main-2}
u_1(x,t)\geq u_{2}(x,t) \qquad \mbox{a.e. in $\R^n\times(0,T)$}
\end{equation}
if $u_{0,1}(x)\geq u_{0,2}(x)$ for all $x\in\R^n$. By \eqref{eq-half-inequality-for-uniqueness-of-solution-of-main-1} and \eqref{eq-half-inequality-for-uniqueness-of-solution-of-main-2}, the lemma follows.
\end{proof}
Let $u$ be a solution of \eqref{eq-for-each-population-u-i} which satisfies \eqref{eq-basic-property-of-u-smaller-than-U-at-zero}. Then by Lemma \ref{lem-uniquness-of-weak-solution-2},
\begin{equation}\label{eq-basic-property-of-u-smaller-than-U}
0\leq u(x,t)\leq U(x,t) \qquad \forall x\in\R^n,\,\,t\geq 0.
\end{equation}

As a consequence of \eqref{eq-basic-property-of-u-smaller-than-U}, we can get the functional space to which the solutions of \eqref{eq-for-each-population-u-i} are belonging.

\begin{lemma}\label{lem-basic-space-where-a-weak-solution-belongs}
Let $m>1$ and let $U$ be the solution of \eqref{eq-PME-satisfied-by-U-total-species} with initial data $U_0\in L^1\left(\R^n\right)\cap L^{1+m}\left(\R^n\right)$ nonnegative and compactly supported. Then solution $u$ of \eqref{eq-for-each-population-u-i} with initial condition \eqref{eq-basic-property-of-u-smaller-than-U-at-zero} satisfies
\begin{equation*}
U^{m-1}\left|\nabla u\right|\in L^2\left(0,T:L^2\left(\R^n\right)\right) .
\end{equation*}
\end{lemma}

\begin{proof}
Multiplying the first equation in \eqref{eq-for-each-population-u-i} by $U^{m-1}u$ and integrating over $\R^n\times(0,\infty)$,  we have
\begin{align}
&\int_{\R^n}U^{m-1}u^2\,dx(t)+\int_{0}^{\infty}\int_{\R^n}m\left(U^{m-1}\left|\nabla u\right|\right)^2\,dxdt\notag\\
&\qquad \qquad \leq \int_{\R^n}U^{m-1}u^2\,dx(0)+\int_{0}^{\infty}\int_{\R^n}m\,u^2\,\left|\nabla U^{m-1}\right|^2\,dxdt+(m-1)\int_{0}^{\infty}\int_{\R^n}U^{m-2}\,u^2\,U_t\,dxdt\notag\\
&\qquad \qquad \leq \int_{\R^n}U_0^{m+1}\,dx+\frac{(m-1)^2}{m}\int_{0}^{\infty}\int_{\R^n}\left|\nabla U^{m}\right|^2\,dxdt+(m-1)\int_{0}^{\infty}\int_{\R^n}U^{m-2}\,u^2\,U_t\,dxdt\label{eq-for-L-2-L-2-of-U-m-1-nabla-u}
\end{align}
by  \eqref{eq-basic-property-of-u-smaller-than-U} and Young's inequality. Since $U$ is the solution of \eqref{eq-PME-satisfied-by-U-total-species},
\begin{align}
\int_{0}^{\infty}\int_{\R^n}U^{m-2}\,u^2\,U_t\,dxdt&=-\int_{0}^{\infty}\int_{\R^n}\nabla\left(U^{m-2}\,u^2\right)\cdot\nabla U^m\,dxdt\notag\\
&\leq \frac{\left|m-2\right|}{m}\int_{0}^{\infty}\int_{\R^n}\left|\nabla U^m\right|^2\,dxdt+2\int_{0}^{\infty}\int_{\R^n}U^{m-1}\left|\nabla u\right|\left|\nabla U^m\right|\,dxdt\notag\\
&\leq \left(\frac{\left|m-2\right|}{m^2}+2\right)\int_{0}^{\infty}\int_{\R^n}\left|\nabla U^m\right|^2\,dxdt+\frac{1}{2}\int_{0}^{\infty}\int_{\R^n}\left(U^{m-1}\left|\nabla u\right|\right)^2\,dxdt.\label{eq-controlling-a-term-in-eq-for-L-2-L-2-of-U-m-1-nabla-u}
\end{align}
By \eqref{eq-for-L-2-L-2-of-U-m-1-nabla-u} and \eqref{eq-controlling-a-term-in-eq-for-L-2-L-2-of-U-m-1-nabla-u},
\begin{equation*}
\int_{0}^{\infty}\int_{\R^n}\left(U^{m-1}\left|\nabla u\right|\right)^2\,dxdt\leq C\left(\left\|U_0\right\|_{L^{1+m}\left(\R^n\right)},\left\|\nabla U^m\right\|_{L^2\left(0,\infty:L^2\left(\R^n\right)\right)}\right)<\infty
\end{equation*}
and the lemma follows.
\end{proof}

We now are ready for the existence of weak solution of \eqref{eq-for-each-population-u-i}.

\begin{lemma}\label{lem-existence-of-soluiton-u-by-u-epsilon-M-weak-convergence}
Let $m>1$ and let $U$ be the solution of \eqref{eq-PME-satisfied-by-U-total-species} with initial data $U_0\in L^1\left(\R^n\right)\cap L^{1+m}\left(\R^n\right)$ nonnegative and compactly supported. Let  $u_0\in L^1\left(\R^n\right)$ be a function with $0\leq u_0\leq U_0$,  Then there exists a weak solution $u$ of \eqref{eq-for-each-population-u-i} which satisfies \eqref{eq-basic-property-of-u-smaller-than-U}.
\end{lemma}
\begin{proof}
For the functions $u_0$, $U$ and constants $M>1$, $0<\epsilon<1$, let
\begin{equation*}
\begin{cases}
\begin{aligned}
u_{0,M}(x,t)&=\min\left(u_0(x),M\right)\\
U_{M}(x,t)&=\min\left(U(x,t),M\right)
\end{aligned}
\end{cases}
\end{equation*}
Then, for any $0<\epsilon<1$, $M>1$ there exists the solution $u_{\epsilon,M}$ of 
\begin{equation}\label{eq-PDE-with-Uniformly-parabolic-U-epsilon-M-1}
\begin{cases}
\begin{aligned}
\left(u_{\epsilon,M}\right)_{t}&=\nabla\left(m\,\left(U_{M}^{m-1}+\epsilon\right)\nabla u_{\epsilon,M}\right) \qquad \mbox{in $\R^n\times\left(0,\infty\right)$}\\
u_{\epsilon,M}(x,0)&=u_{0,M}(x) \qquad \qquad \forall x\in\R^n.
\end{aligned}
\end{cases}
\end{equation}
Multiplying the first equation in \eqref{eq-PDE-with-Uniformly-parabolic-U-epsilon-M-1} by $u_{\epsilon,M}$ and integrating over $\R^n\times\left(0,\infty\right)$, we have
\begin{align}
\int_{\R^n}\left|u_{\epsilon,M}\right|^2dx(t)&+\int_{t}^{\infty}\int_{\R^n}U_M^{m-1}\left|\nabla u_{\epsilon,M}\right|^2\,dxdt\notag\\
&+\int_t^{\infty}\int_{\R^n}\left(\epsilon^{\frac{1}{2}}\left|\nabla u_{\epsilon,M}\right|\right)^2\,dxdt\leq C\left(\left\|u_0\right\|_{L^2},m\right), \qquad \forall t>0. \label{eq-weighted-bound}
\end{align}
By Lemma \ref{lem-cf-chapter-9-of-Va1}, for any $M>0$ there exists a constant $t_{M}>0$ such that
\begin{equation}\label{eq-time-t-M-which-equal-U=and-U-M}
t_M\to 0 \quad \mbox{as $M\to\infty$} \qquad \mbox{and} \qquad U_M(x,t)=U(x,t) \quad \forall x\in\R^n,\,\,t\geq t_M.
\end{equation}
By \eqref{eq-weighted-bound} and \eqref{eq-time-t-M-which-equal-U=and-U-M}, there exists a some function $u$ such that
\begin{equation}\label{eqw-convergnece-as-epsilon-and-M-to-0-and-infty-in-L-2-loc}
\begin{cases}
\begin{aligned}
u_{\epsilon,M}&\to u \qquad \qquad \qquad \mbox{in $L^2_{loc}\cap L^{1+m}_{loc}$}\\
U^{\frac{m-1}{2}}\left|\nabla u_{\epsilon,M}\right|&\to U^{\frac{m-1}{2}}\left|\nabla u\right| \qquad \mbox{in $L^2_{loc}$}\\
\epsilon\left|\nabla u_{\epsilon,M}\right|&\to 0 \qquad \qquad \qquad \mbox{in $L^2_{loc}$}
\end{aligned}
\end{cases}
\end{equation}
as  $\epsilon\to 0$ and $M\to \infty$. Choosing $\vp\in C^{2,1}_0\left(\R^n\times(0,\infty)\right)$, multiplying it to the first equation of \eqref{eq-PDE-with-Uniformly-parabolic-U-epsilon-M-1}, and integrating over $\R^n\times(0,\infty)$, we have
\begin{equation}\label{eq-weak-formula-of-u-epsilon-M-09}
\int_0^{\infty}\int_{\R^n}\left\{m\,U^{m-1}\nabla u_{\epsilon,M}\cdot\nabla\vp+m\epsilon\nabla u_{\epsilon,M}\nabla\vp-u_{\epsilon,M}\vp_t\right\}\,dxdt=0
\end{equation}
for sufficiently large $M>0$. Letting $\epsilon\to0$ and then $M\to\infty$ in \eqref{eq-weak-formula-of-u-epsilon-M-09}, by \eqref{eqw-convergnece-as-epsilon-and-M-to-0-and-infty-in-L-2-loc} $u$ satisfies
\begin{equation}\label{eq-weak-formula-of-u-0089}
\int_0^{\infty}\int_{\R^n}\left\{m\,U^{m-1}\nabla u\cdot\nabla\vp-u\vp_t\right\}\,dxdt=0 \qquad \forall \vp\in C^{2,1}_0\left(\R^n\times(0,\infty)\right).
\end{equation}
\indent We now are going to show that 
\begin{equation}\label{eq-converges-of-u-to-initial-data-as-t-to-zero}
u\left(\cdot,t\right)\to u_0\qquad  \mbox{in $L^1$ as $t\to 0^+$}.
\end{equation} 
Let $\eta(x)\in C_0^{2}\left(\R^n\right)$ and $0<t<1$. Multiply the first equation of \eqref{eq-PDE-with-Uniformly-parabolic-U-epsilon-M-1} by $\eta$, and integrate it over $\R^n\times(0,t)$. Then by an argument similar to the proof of Lemma \ref{lem-basic-space-where-a-weak-solution-belongs} we have
\begin{equation}\label{compare-between-u-and-u-0-with-eta23579}
\begin{aligned}
&\left|\int_{\R^n}u_{\epsilon,M}(x,t)\eta(x)\,dx-\int_{\R^n}u_{0,\epsilon,M}(x)\eta(x)\,dx\right|\notag\\
&\qquad \qquad  \leq \int_{0}^{t}\int_{\R^n}U^{m-1}_{\epsilon,M}(x,t)\left|\nabla u_M(x,t)\right|\left|\nabla\eta(x)\right|\,dxdt\\
&\qquad \qquad \leq C\left(\left\|U_0\right\|_{L^1\left(\R^n\right)},\left\|\nabla U^m\right\|_{L^2\left(\R^n\times(0,1)\right)},\left\|\nabla\eta\right\|_{L^{\infty}}\right)\sqrt{t} \qquad \forall 0<t<1
\end{aligned}
\end{equation}
Letting  $\epsilon\to 0$, $M\to\infty $ and then $t\to 0$ in \eqref{compare-between-u-and-u-0-with-eta23579}, the claim follows. Therefore $u$ is a weak solution of \eqref{eq-for-each-population-u-i} which satisfies \eqref{eq-basic-property-of-u-smaller-than-U} and the lemma follows.
\end{proof}

\subsection{Equivalence properties on $u$ and $U$}
Since the equations satisfied by $u$ and $U$ have the the same diffusion coefficients $U^{m-1}$, it can be expect that the solutions of \eqref{eq-for-each-population-u-i}  and \eqref{eq-PME-satisfied-by-U-total-species} have a lot things in common. By an argument similar to the proof of 9.15 of \cite{Va1}, we  have an important conservation.
\begin{lemma}\label{lem-Mass-conservation-of-u}
For the solution $U$ of \eqref{eq-PME-satisfied-by-U-total-species} with initial data $U_0\in L^1\left(\R^n\right)\cap L^{1+m}\left(\R^n\right)$ nonnegative and compactly supported, let $u$ be a weak solution of \eqref{eq-for-each-population-u-i}. Then, for every $t>0$ we have
\begin{equation*}
\int_{\R^n}u(x,t)\,dx=\int_{\R^n}u_{0}\,dx.
\end{equation*}
\end{lemma}
\begin{proof}
Let $\left\{\xi_{l}(x)\right\}_{l=1}^{\infty}\subset C^{\infty}(\R^n)$ be a sequence of functions such that $\xi_l(x)=1$ for $|x|\leq l-1$, $\xi_l(x)=0$ for $|x|\geq l$ and $0<\xi_l<1$ for $l-1<|x|<l$. Multiplying the first equation in \eqref{eq-for-each-population-u-i} by $\xi_l$ and integrating, we have
\begin{equation*}
\begin{aligned}
&\int_{\R^n}u(x,t)\xi_l(x)\,dx-\int_{\R^n}u_{0}(x,t)\xi_l(x)\,dx\\
&\qquad \qquad =\int_{0}^{t}\int_{\R^n}\left(u\right)_{\tau}\xi_l\,dxd\tau\\
&\qquad \qquad =-\int_0^{t}\int_{\R^n}m\,U^{m-1}(x,\tau)\left(\nabla u(x,\tau)\cdot\nabla\xi_l(x)\right)\,dxd\tau.
\end{aligned}
\end{equation*}
Then by Lemma \ref{lem-basic-space-where-a-weak-solution-belongs},
\begin{equation}\label{eq-absolute-value-of-difference-of-u-0-and-u-t-in-L-1}
\begin{aligned}
&\left|\int_{\R^n}u(x,t)\xi_l(x)\,dx-\int_{\R^n}u_{0}(x,t)\xi_l(x)\,dx\right|\\
&\qquad \qquad \leq m\left\|\nabla\xi_l\right\|_{L^{\infty}}\left(\int_0^{t}\int_{B_{l}\bs B_{l-1}}\left|U^{m-1}\nabla u\right|^2\,dxd\tau\right)^{\frac{1}{2}} \to 0 \qquad \mbox{as $l\to\infty$}
\end{aligned}
\end{equation}
and the lemma follows.
\end{proof}

On any compact subset of the region where $U>0$, the equation for $u$ becomes non-degenerate parabolic equation. Then by standard theory for non-degenerate parabolic equation \cite{LSU}, the solution  $u$ can be immediately positive on that region if the solution $u$ is strictly positive at a point of that region. As a consequence of this expectation, we can have the following equivalence between solutions $U$ and $u$.

\begin{lemma}\label{eq-same-support-between-U-and-u}
Let $m>1$ and $t_0\geq 0$. Let $U$ be the solution of \eqref{eq-PME-satisfied-by-U-total-species} with initial data $U_0\in L^1\left(\R^n\right)\cap L^{1+m}\left(\R^n\right)$ nonnegative and compactly supported. Suppose that $u\geq 0$ satisfies 
\begin{equation}\label{eq-for-solution-u-for-compact-support-comparing-with-U}
u_t=\nabla\left(U^{m-1}\nabla u\right) \qquad  \mbox{ in the distribution sense in $\R^n\times(t_0,\infty)$}
\end{equation} 
and \eqref{eq-basic-property-of-u-smaller-than-U}. Then 
\begin{equation}\label{eq-equal-of-support-after-some-times-flow}
\textbf{supp} \,U(t)=\textbf{supp}\,u(t) \qquad \forall t> t_0.
\end{equation}
\end{lemma}
\begin{proof}
By \eqref{eq-basic-property-of-u-smaller-than-U}, we first have
\begin{equation*}
\textbf{supp}\,u(t)\subset \textbf{supp}\,U(t) \qquad \forall t\geq t_0.
\end{equation*}
We now suppose that \eqref{eq-equal-of-support-after-some-times-flow} fails for some $t_1>t_0$. Then, there exists a point $x_0\in\partial\,\textbf{supp}\,u(t_1)$ such that
\begin{equation}\label{eq-existence-of-cylinder-where-diff-coeff-U-strictly-positive}
B_{2r}\left(x_0\right)\subset \textbf{supp}\,U(s) \qquad \forall t_1-2\epsilon_1\leq s\leq t_1
\end{equation}
for sufficiently small $r>0$ and $0<\epsilon_1<\frac{1}{2}\left(t_1-t_0\right)$. By \eqref{eq-existence-of-cylinder-where-diff-coeff-U-strictly-positive},  the diffusion coefficients of \eqref{eq-for-solution-u-for-compact-support-comparing-with-U} is uniformly parabolic in $B_{2r}(x_0)\times[t_1-2\epsilon_1,t_1]$. Thus by standard theory for non-degenerate parabolic equation \cite{LSU}, the solution $u$ is continuous on $B_{r}(x_0)\times\left[t_1-\epsilon_1,t_1\right]$. This implies that
\begin{equation}\label{eq-property-of-u-not-equivalent-to-zero-near-x-0-t-1}
u\left(\cdot,t\right)\not\equiv 0 \qquad \mbox{on $B_{r}(x_0)$} \qquad \forall t\in\left[t_1-\epsilon_1,t_1\right]
\end{equation}
for sufficiently small $\epsilon_1>0$.\\
\indent For $0<\tau<\epsilon_1$, let $v_{0,\tau}(x)=u_i\left(x,t_1-\tau\right)\chi_{B_{r}(x_0)}$. Then by \eqref{eq-existence-of-cylinder-where-diff-coeff-U-strictly-positive}, there exists an unique solution $v^{\tau}$ of 
\begin{equation*}
\begin{cases}
\begin{aligned}
v_t(x,t)&=\nabla\left(m\,U^{m-1}(x,t+t_1-\tau)\nabla v(x,t)\right) \qquad \mbox{in $B_{r}(x_0)\times(0,\tau)$}\\
v(x,t)&=0 \qquad \qquad \qquad \qquad \qquad \qquad \qquad \mbox{on $\partial B_{r}(x_0)\times(0,\tau)$}\\
v(x,0)&=v_{0,\tau}(x) \qquad \qquad \qquad \qquad \qquad \qquad \mbox{in $B_{r}(x_0)$}.
\end{aligned}
\end{cases}
\end{equation*}
In addition, by \eqref{eq-property-of-u-not-equivalent-to-zero-near-x-0-t-1} and standard theory for non-degenerate parabolic equation \cite{LSU}, there exists a constant $c_1>0$ such that
\begin{equation}\label{eq-lower-bound-of-v-tau-at-t-1}
v^{\tau}(x,\tau)\geq c_1\qquad \forall x\in B_{\frac{r}{2}}(x_0).
\end{equation}
Since $u(x,t+t_1-\tau)$ is also a solution with initial data $u(x,t_1-\tau)$ which is bigger than $v_{0,\tau}(x)$ in $B_{r}(x_0)$, by \eqref{eq-lower-bound-of-v-tau-at-t-1} and the comparison principle we have
\begin{equation*}
u(x_0,t_1)\geq v^{\tau}(x_0,\tau)\geq c_1>0.
\end{equation*}
This contradicts the fact that $u(x_0,t_1)=0$. Hence \eqref{eq-equal-of-support-after-some-times-flow} holds for all $t\geq t_0$ and the lemma follows.
\end{proof}

\section{Local Continuity}
\setcounter{equation}{0}
\setcounter{thm}{0}

This section will be devoted to prove the local continuity of solution $u$ of \eqref{eq-standard-form-for-local-continuity-estimates-intro-1} under the  \textbf{Assumption I and II}. As an application, we also deal with the local H\"older continuity of solution of \eqref{eq-pme_system}.\\ Before  moving on to the main steps, we will give an example about the boundedness of  the function $U$ in diffusion coefficients of \eqref{eq-standard-form-for-local-continuity-estimates-intro-1}. 
\begin{lemma}\label{eq-example-boundedness-of-diffusion-coefficients-of-generaized-equation}
Let $1<m<\frac{n+2}{n-2}$ and let $U$ be a function such that
\begin{equation*}
U\in L^{\infty}\left(0,T;L^2\left(\R^n\right)\right) \qquad \mbox{and} \qquad U^m\in L^2\left(0,T;H^1\left(\R^n\right)\right).
\end{equation*}
Let measurable functions $\mathcal{A}$ and $\mathcal{B}$ be given by \eqref{eq-condition-1-for-measurable-functions-mathcal-A-and-mathcal-B}-\eqref{eq-condition-3-for-measurable-functions-mathcal-A-and-mathcal-B} for any constant $1\leq b<\frac{2n}{n-2}$. If $U$ is a subsolution of \eqref{eq-standard-form-for-local-continuity-estimates-intro-1}, i.e.,
\begin{equation}\label{eq-form-for-subsolution-local-continuity-estimates-section-3}
U_t\leq \nabla\cdot\left(mU^{m-1}\mathcal{A}\left(\nabla U,U,x,t\right)+\mathcal{B}\left(U,x,t\right)\right),
\end{equation}
then there exists a constant $C(T)>0$ such that
\begin{equation*}
\sup_{x\in\R^n}\left|U\left(x,T\right)\right|\leq C\left(T\right).
\end{equation*}
\end{lemma}
\begin{proof}
We will use a modification of the proof of Theorem 1 of \cite{CV}. Let 
\begin{equation*}
L_j=M\left(1-\frac{1}{2^{j}}\right) \qquad \mbox{and} \qquad U_j=\left(U-L_j\right)_+
\end{equation*}
for a constant $M$ which will be determined later. Multiplying \eqref{eq-form-for-subsolution-local-continuity-estimates-section-3} by $U_j$ and integrating it over $\R^n$, we have
\begin{equation}\label{eq-energy-type-inequality-for-boundedness-of-U}
\begin{aligned}
\frac{\partial}{\partial t}\left[\int_{\R^n}U_j^2\,dx\right]+\int_{\R^n}\left|\nabla U_j\right|^2\,dx&\leq \frac{\partial}{\partial t}\left[\int_{\R^n}U_j^2\,dx\right]+\left(mC_1\left(\frac{M}{2}\right)^{m-1}-1\right)\int_{\R^n}\left|\nabla U_j\right|^2\,dx\\
&\leq \int_{\R^n}\left(U^{m+1}f_1+U^{2b}f_3^2\right)\chi_{_{\left\{U_j\geq 0\right\}}}\,dx\\
&\leq \left(\left\|f_1\right\|_{L^{\infty}}+\left\|f_3\right\|^2_{L^{\infty}}\right)\int_{\R^n}\left(U^{m+1}+U^{2b}\right)\chi_{_{\left\{U_j\geq 0\right\}}}\,dx
\end{aligned}
\end{equation}
for any $M\geq\frac{2^{\frac{m}{m-1}}}{\left(mC_1\right)^{\frac{1}{m-1}}}$. For fixed $t_0>0$, let $T_j=t_0\left(1-\frac{1}{2^j}\right)$ and 
\begin{equation*}
A_j=\sup_{t\geq T_j}\left(\int_{\R^n}U_j^2\,dx\right)+\frac{1}{2}\int_{T_j}^{\infty}\int_{\R^n}\left|\nabla U_j\right|^2\,dx.
\end{equation*}
Integrating \eqref{eq-energy-type-inequality-for-boundedness-of-U} over $\left(s,t\right)$ and $\left(s,\infty\right)$, $\left(T_{j-1}<s<T_j,\,\,t>T_j\right)$, we have
\begin{equation}\label{eq-inequality-with-A-j-1}
A_j+\frac{1}{2}\int_{T_j}^{\infty}\int_{\R^n}\left|\nabla U_j\right|^2\,dxdt\leq \int_{\R^n}U_j^2\left(x,s\right)\,dx+C_1\int_{T_{j-1}}^{\infty}\int_{\R^n}\left(U^{m+1}+U^{2b}\right)\chi_{_{\left\{U_j\geq 0\right\}}}\,dxdt
\end{equation}
for some constant $C_1>0$. By mean value theorem for integration, we find
\begin{equation}\label{eq-inequality-with-A-j-2}
A_j+\frac{1}{2}\int_{T_j}^{\infty}\int_{\R^n}\left|\nabla U_j\right|^2\,dxdt\leq \frac{2^{j}}{t_0}\int_{T_{j-1}}^{\infty}\int_{\R^n}U_j^2\,dxdt+C_1\int_{T_{j-1}}^{\infty}\int_{\R^n}\left(U^{m+1}+U^{2b}\right)\chi_{_{\left\{U_j\geq 0\right\}}}\,dxdt
\end{equation}
Observe that 
\begin{equation}\label{eq-lower-bound-of-U-j-1-by-M-over-2-j-on-U-j}
U_{j-1}\geq \frac{M}{2^j}\qquad \mbox{ if $U_j>0$.}
\end{equation}
Then
\begin{equation}\label{eq-replacing-U-by-U-sub-j-with-2-to-i}
U=U_{j-1}+L_{j-1}\leq U_{j-1}+M\qquad \Rightarrow \qquad 2^{j+1}U_{j-1}\geq U\qquad \mbox{ if $U_j>0$.}
\end{equation}
By \eqref{eq-inequality-with-A-j-2} and \eqref{eq-replacing-U-by-U-sub-j-with-2-to-i},
\begin{equation}\label{eq-inequality-with-A-j-3}
A_j+\frac{1}{2}\int_{T_j}^{\infty}\int_{\R^n}\left|\nabla U_j\right|^2\,dxdt\leq \frac{2^{j}}{t_0}\int_{T_{j-1}}^{\infty}\int_{\R^n}U_j^2\,dxdt+C_14^{j}\int_{T_{j-1}}^{\infty}\int_{\R^n}\left(U_{j-1}^{m+1}+U_{j-1}^{2b}\right)\chi_{_{\left\{U_j\geq 0\right\}}}\,dxdt.
\end{equation}
By Sobolev and Interpolation inequalities of $L^p$-space, there exist constant $C_2$ and $C_3>0$ such that
\begin{equation}\label{eq-Sobolev-and-Interpolation-inequalities-for-replacign-U-by-U-j}
\begin{aligned}
C_1\int_{\R^n}\left(U_{j-1}^{m+1}+U_{j-1}^{2b}\right)\chi_{_{\left\{U_j\geq 0\right\}}}\,dxdt&\leq C_2\left(\int_{T_j}^{\infty}\int_{\R^n}U_{j-1}^{\frac{2n}{n-2}}\chi_{_{\left\{U_j\geq 0\right\}}}\,dxdt\right)^{\frac{n-2}{n}}+C_3\int_{T_j}^{\infty}\int_{\R^n}U_{j-1}^2\,dxdt\\
&\leq \frac{1}{2}\int_{T_j}^{\infty}\int_{\R^n}\left|\nabla U_{j}\right|^2\,dxdt+C_3\int_{T_j}^{\infty}\int_{\R^n}U_{j-1}^2\,dxdt.
\end{aligned}
\end{equation}
By \eqref{eq-inequality-with-A-j-3} and \eqref{eq-Sobolev-and-Interpolation-inequalities-for-replacign-U-by-U-j}, 
\begin{equation}\label{eq-inequality-with-A-j-4}
A_j\leq C_44^j\int_{T_{j-1}}^{\infty}\int_{\R^n}U_{j-1}^2\chi_{_{\left\{U_j\geq 0\right\}}}\,dxdt
\end{equation}
for some constant $C_4=C_4\left(t_0\right)>0$. By \eqref{eq-lower-bound-of-U-j-1-by-M-over-2-j-on-U-j}, we also have
\begin{equation}\label{eq-upper-bound-of-1-on-U-j-positive-3}
\chi_{_{\left\{U_j\geq 0\right\}}}\leq \left(\frac{2^j}{M}U_{j-1}\right)^{\frac{4}{n}}
\end{equation}
By \eqref{eq-inequality-with-A-j-4}, \eqref{eq-upper-bound-of-1-on-U-j-positive-3}, Sobolev and H\"older inequalities,
\begin{equation*}
\begin{aligned}
A_j\leq C_4\frac{4^{\frac{n+2}{n}j}}{M^{\frac{4}{n}}}\int_{T_{j-1}}^{\infty}\int_{\R^n}U_{j-1}^{2\left(\frac{n+2}{n}\right)}\,dxdt\leq C_5\frac{4^{\frac{n+2}{n}j}}{M^{\frac{4}{n}}}A_{j-1}^{\frac{n+2}{n}}
\end{aligned}
\end{equation*}
for some constant $C_5=C_5\left(t_0\right)>0$. If we choose the constant $M>2$ sufficiently large that
\begin{equation*}
A_1\leq \left(\frac{C_5}{M^{\frac{4}{n}}}\right)^{-\frac{n}{2}}\left(4^{\frac{n+2}{n}}\right)^{-\left(\frac{n}{2}\right)^2},
\end{equation*}
then $A_j\to 0$ as $j\to\infty$, i.e.,
\begin{equation*}
\sup_{x\in\R^n}\left|U\left(x,t_0\right)\right|\leq M=M\left(t_0\right) 
\end{equation*} 
and the lemma follows.
\end{proof}

For the local continuity of solution of \eqref{eq-standard-form-for-local-continuity-estimates-intro-1}, we start by stating well-known result, Sobolev-type inequality.

\begin{lemma}[cf. Lemma 3.1 of \cite{KL1}]\label{eq-Sobolev-Type-Inequality}
Let $\eta(x,t)$ be a cut-off function compactly supported in $B_r$ and let $u$ be a function defined in $\R^n\times(t_1,t_2)$ for any $t_2>t_1>0$. Then $u$ satisfies the following Sobolev inequalities:
\begin{equation}\label{eq-first-sobolev-inequality-1}
\left\|\eta u\right\|_{L^{\frac{2n}{n-2}}(\R^n)}\leq C\left\|\nabla (\eta u)\right\|_{L^2(\R^n)}
\end{equation}
and
\begin{equation}\label{eq-weighted-sobolev-inequality-for-holder}
\begin{aligned}
\|\eta\, u\|^2_{L^2(t_1,t_2;L^{2}(\R^n))}\leq C\Big(\sup_{t_1\leq t\leq t_2}\|\eta\, u\|^2_{L^{2}(\R^n)}+\|\D
(\eta\, u)\|^2_{L^2(t_1,t_2;L^2(\R^n))}\Big)\,
|\{\eta\, u>0\}|^{\frac{2}{n+2}}
\end{aligned}
\end{equation}
for some $C>0$.
\end{lemma}

From now on, we are going to focus on oscillation argument. To apply it to our case, we use a modification of the technique introduced in \cite{Di}, \cite{KL1}, \cite{HU}. \\
\indent Choose a point $(x_0,t_0)\in\R^n\times(0,\infty)$ and a constant $R_0>0$ such that
\begin{equation*}
\left(x_0,t_0\right)+Q\left(R_0,R_0^{2-\epsilon}\right)\subset \R^n\times(0,\infty).
\end{equation*} 
where $\epsilon>0$ is a small number which is determined by \eqref{eq-condition-of-epsilon-for-holder-continuity-at-start}.
After translation, we may assume without loss of generality that 
\begin{equation*}
(x_0,t_0)=(0,0). 
\end{equation*}
Set
\begin{equation*}
\mu^+=ess\sup_{Q(R_0,R_0^{2-\epsilon})}u, \qquad \mu^-=ess\inf_{Q(R_0,R_0^{2-\epsilon})}u, \qquad \omega=\osc_{Q(R_0,R_0^{2-\epsilon})} u=\mu^+-\mu^-.
\end{equation*}
By the \textbf{Assumption I}, the equation \eqref{eq-standard-form-for-local-continuity-estimates-intro-1} is non-degenerate on the region where $u>0$. Thus if $\mu^->0$, then the equation is uniformly parabolic in $Q\left(R_0,R_0^{2-\epsilon}\right)$. By standard regularity theory for the parabolic equation \cite{LSU}, the local H\"older continuity follows. Hence from now on, we assume that 
\begin{equation*}
\mu^-=0.
\end{equation*}
If $\mu^+=0$, then
\begin{equation*}
u\equiv 0 \qquad \mbox{on $Q\left(R_0,R_0^{2-\epsilon}\right)$}.
\end{equation*}
This immediately implies the local H\"older continuity of solution $u$. Hence we also assume that
\begin{equation*}
\omega=\mu^+>0.
\end{equation*}
Construct the cylinder 
\begin{equation}\label{eq-construction-of-cylinder-with-some-constant-A-which-is-bigger}
Q\left(R,\theta_0^{-\alpha_0}R^2\right)=B_{R}\times\left(-\theta_0^{-\alpha_0}R^2,0\right) \qquad \left(\theta_0=\frac{\omega}{4},\,\,\alpha_0=\beta(m-1)\right)
\end{equation}
where $\beta$ is given by \eqref{eq-assumption-through-the-regularity-section-0}. If $U$ is uniformly parabolic, then the constant $\beta$ is zero. Thus the scaled parabolic cylinder $Q\left(R,\theta_0^{-\alpha_{0}}R^2\right)$ is equivalent to the standard cylinder $Q(R,R^2)$ with homogeneous of degree one. Therefore De Giorgi and Moser's technique \cite{De}, \cite{Mo} on regularity  theory for uniformly elliptic and parabolic  PDE's are enough to show the local continuity of solution $u$ of  \eqref{eq-standard-form-for-local-continuity-estimates-intro-1}. Otherwise, $\theta_0^{\alpha_{0}}$ depends on the size of oscillation $\omega$. Since it depends on the value of solution $u$ at each $(x,t)\in\R^n\times\R^+$,  we will use the intrinsic scaling technique to overcome the difficulties on local continuity stem from the relation between $u$ and $U$.\\
\indent We will assume that the radius $0<R<R_0$ is sufficiently small that
\begin{equation}\label{eq-condition-between-R-and-theta-alpha--1}
\theta_0^{\alpha_{0}}>R^{\epsilon}.
\end{equation}
By \eqref{eq-construction-of-cylinder-with-some-constant-A-which-is-bigger} and \eqref{eq-condition-between-R-and-theta-alpha--1}, 
\begin{equation*}
Q\left(R,\theta_0^{-\alpha_{0}}R^2\right)\subset Q\left(R,R^{2-\epsilon}\right)\subset Q\left(R_0,R_0^{2-\epsilon}\right)
\end{equation*}
and
\begin{equation*}
\osc_{Q\left(R,\theta_0^{-\alpha_{0}}R^2\right)}u\leq \omega=4\theta.
\end{equation*}
\indent To take care of the regularity problem in $u_t$, we introduce the Lebesgue-Steklov average $u_h$ of the weak solution $u$, for $h>0$:
\begin{equation*}
u_h(\cdot,t)=\frac{1}{h}\int_{t}^{t+h}u(\cdot,\tau)\,d\tau.
\end{equation*}
$u_h$ is well-defined and it converges to $u$ as $h\to 0$ in $L^{p}$ for all $p\geq 1$.  In addition, it is differentiable in time for all $h>0$ and its derivative is 
\begin{equation*}
\frac{u(t+h)-u(t)}{h}.
\end{equation*}
\indent Fix $t\in(0,T)$ and let $h$ be a small positive number such that $0<t<t+h<T$. Then for every compact subset $\mathcal{K}\subset\R^n$ the following formulation is equivalent to \eqref{eq-standard-form-for-local-continuity-estimates-intro-1}.
\begin{equation}\label{eq-formulation-for-weak-solution-of-u-h}
\int_{\mathcal{K}\times\{t\}}\left[\left(u_h\right)_t\vp+m\left(U^{m-1}\mathcal{A}\left(\nabla u,u,x,t\right)\right)_h\nabla\vp+\left(\mathcal{B}\left(u,x,t\right)\right)_h\nabla\vp\right]\,dx=0, \qquad \forall 0<t<T-h
\end{equation}
for any $\vp\in H^1_0\left(\mathcal{K}\right)$.
\subsection{The First Alternative}
We now start by stating the first alternative.
\begin{lemma}\label{lem-the-first-alternative-for-holder-estimates}
There exists a positive number $\rho_0$ depending on $\frac{\Lambda}{\theta_0^{\,\beta}}$ such that if
\begin{equation}\label{eq-first-condition-of-small-region-of-lower-for-holder-estimates}
\left|\left\{\left(x,t\right)\in Q\left(R,\theta_0^{-\alpha_{0}}R^2\right):u(x,t)<\frac{\omega}{2}\right\}\right|\leq\rho_0\left|Q\left(R,\theta_0^{-\alpha_{0}}R^2\right)\right|
\end{equation}
then,
\begin{equation}\label{eq-conclusion-of-first-alternatives-half-of-difference}
u(x,t)>\frac{\omega}{4} \qquad \mbox{for all $(x,t)\in Q\left(\frac{R}{2},\theta_0^{-\alpha_{0}}\left(\frac{R}{2}\right)^2\right)$}.
\end{equation}
\end{lemma}
\begin{proof}
For $i\in\N$, we set
\begin{equation*}
R_i=\frac{R}{2}+\frac{R}{2^{i}} \qquad \mbox{and} \qquad l_i=\mu_-+\left(\frac{\omega}{4}+\frac{\omega}{2^{i+1}}\right)=\frac{\omega}{4}+\frac{\omega}{2^{i+1}}.
\end{equation*}
Consider a cut-off function $\eta_i(x,t)\in C^{\infty}\left(\R^n\times\R\right)$ such that
\begin{equation*}
\begin{cases}
\begin{array}{cccl}
0\leq \eta_i\leq 1 &&& \mbox{in $Q\left(R_i,\theta_0^{-\alpha_{0}}R_i^2\right)$}\\
\eta_i=1 &&& \mbox{in$Q\left(R_{i+1},\theta_0^{-\alpha_{0}}R_{i+1}^2\right)$ }\\
\eta_i=0 &&& \mbox{on the parabolic boundary of $Q\left(R_i,\theta_0^{-\alpha_{0}}R_i^2\right)$}\\
\left|\nabla\eta_i\right|\leq \frac{2^{i+1}}{R_i},\,\,\left|\left(\eta_i\right)_t\right|\leq \frac{2^{2(i+1)}\theta_0^{\alpha_0}}{R_i^2}&&& \mbox{in $Q\left(R_i,\theta_0^{-\alpha_0}R_i^2\right)$}
\end{array}
\end{cases}
\end{equation*}
In the weak formulation \eqref{eq-formulation-for-weak-solution-of-u-h}, we take $\vp=\left(u_h-l_i\right)_-\eta_i^2$ and integrate over $\left(-\theta_0^{-\alpha_{0}}R_i^2,t\right)$ for $t\in\left(-\theta_0^{-\alpha_{0}}R_i^2,0\right)$. Then 
\begin{equation}\label{eq-energy-type-equation-with-Lebesgue-Steklov-everage}
\begin{aligned}
0&=\int_{-\theta_0^{-\alpha_{0}}R_i^2}^{t}\int_{B_{R_i}}\left(u_h\right)_t\left[\left(u_h-l_i\right)_-\eta_i^2\right]\,dxd\tau\\
&\qquad\qquad  +m\int_{-\theta_0^{-\alpha_{0}}R_i^2}^{t}\int_{B_{R_i}}\left(\mathcal{A}\left(\nabla u,u,x,t\right)\right)_h\nabla\left[\left(u_h-l_i\right)_-\eta_i^2\right]\,dxd\tau\\
&\qquad \qquad \qquad\qquad+\int_{-\theta_0^{-\alpha_{0}}R_i^2}^{t}\int_{B_{R_i}}\left(\mathcal{B}\left(u,x,t\right)\right)_h\nabla\left[\left(u_h-l_i\right)_-\eta_i^2\right]\,dxd\tau\\
&:=I+II+III.
\end{aligned}
\end{equation}
Letting $h\to 0$ in \eqref{eq-energy-type-equation-with-Lebesgue-Steklov-everage}, we have
\begin{equation}\label{eq-controlling-the-first-opart-time-derivatives-to-sup}
\begin{aligned}
&-I\geq \frac{1}{2}\int_{B_{R_i}\times\left\{t\right\}}\left(u-l_i\right)_-^2\eta_i^2\,dx-\frac{2^{2(i+1)}\theta_0^{\alpha_{0}}\left(2\theta_0\right)^2}{R_i^2}\int_{-\theta_0^{-\alpha_{0}}R_i^2}^{t}\int_{B_{R_i}}\chi_{\left[u\leq l_i\right]}\,dxd\tau
\end{aligned}
\end{equation}
and, by  Young's inequality
\begin{equation*}
\begin{aligned}
-II&\geq\frac{mC_1}{2}\int_{-\theta_0^{-\alpha_{0}}R_i^2}^{t}\int_{B_{R_i}}U^{m-1}\left|\nabla\left(u-l_i\right)_-\right|^2\eta_i^2\,dxd\tau\\
&\qquad \qquad-\frac{mC_2^22^{2i+2}\Lambda^{m-1}\left(2\theta_0\right)^2}{C_1R_i^2}\int_{-\theta_0^{-\alpha_{0}}R_i^2}^{t}\int_{B_{R_i}}\chi_{\left[u\leq l_i\right]}\,dxd\tau\\
&\qquad \qquad \qquad \qquad -2m\int_{-\theta_0^{-\alpha_{0}}R_i^2}^{t}\int_{B_{R_i}}U^{m-1}u^2\chi_{_{\left[u\leq l_j\right]}}\left(f_1+f_2\right)\left(\eta_i^2+\eta_i\left|\nabla\eta_i\right|^2+\eta_i\right)\,dxd\tau
\end{aligned}
\end{equation*}
and
\begin{equation}\label{eq-estimates-for-III-forcing-term-2}
\begin{aligned}
III&\leq\frac{mC_1\lambda\theta_0^{\,\alpha_0}}{4}\int_{-\theta_0^{-\alpha_{0}}R_i^2}^{t}\int_{B_{R_i}}\left|\nabla\left(u-l_i\right)_-\right|^2\eta_i^2\,dxd\tau\\
&\qquad \qquad+\frac{1+m}{mC_1\,\lambda\theta_0^{\,\alpha_0}}\int_{-\theta^{-\alpha_{0}}R_j^2}^{t}\int_{B_{R_j}}\,u^{\,2b}\,\chi_{_{\left[u\leq l_j\right]}}f_3^2\eta_j^2\,dxd\tau\\
&\qquad \qquad \qquad \qquad  +\frac{C_1\Lambda^{m-1}}{4}\int_{-\theta^{-\alpha_{0}}R_j^2}^{t}\int_{B_{R_j}}\left(u-l_i\right)_-^2\left|\nabla\eta_j\right|^2\,dxd\tau.
\end{aligned}
\end{equation}
Let $q_1$, $q_2\geq 1$ and $0<\kappa_1<1$ be constants satisfying
\begin{equation}\label{relation-between-contants-q-1-and-q-2-anbd-kappa-1}
\frac{n}{2q_1}+\frac{1}{q_2}=1-\kappa_1
\end{equation}
and let
\begin{equation*}
\widehat{q}=\frac{2q_1\left(1+\kappa\right)}{q_1-1}, \qquad \widehat{r}=\frac{2q_2\left(1+\kappa\right)}{q_2-1} \qquad \mbox{and} \qquad \kappa=\frac{2}{n}\kappa_1.
\end{equation*}
By an argument similar to the proof of Lemma 3.3 of \cite{KL1}, there exists a constant $C>0$ such that
\begin{equation}\label{eq-controlling-the-third-opart-gradient-derivatives-to-v-0}
\begin{aligned}
\int_{-\theta^{-\alpha_{0}}R_j^2}^{t}\int_{B_{R_j}}U^{m-1}u^2\chi_{_{\left[u\leq l_j\right]}}\left(f_1+f_2\right)\left(\eta_i^2+\eta_i\right)\,dxd\tau\leq C\Lambda^{m-1}\theta_0^{\,2}\left(\int_{-\theta^{-\alpha_{0}}R_j^2}^{t}\left|A_{l_j,R_j}^-(\tau)\right|^{\frac{\widehat{r}}{\widehat{q}}}\,d\tau\right)^{\frac{2}{\widehat{r}}\left(1+\kappa\right)}
\end{aligned}
\end{equation}
and
\begin{equation}\label{eq-controlling-the-third-opart-gradient-derivatives-to-v}
\begin{aligned}
\int_{-\theta^{-\alpha_{0}}R_j^2}^{t}\int_{B_{R_j}}\,u^{\,2b}\,\chi_{_{\left[u\leq l_j\right]}}f_3^2\eta_j^2\,dxd\tau\leq C\Lambda^{2\left(b-1\right)}\theta_0^{\,2}\left(\int_{-\theta^{-\alpha_{0}}R_j^2}^{t}\left|A_{l_j,R_j}^-(\tau)\right|^{\frac{\widehat{r}}{\widehat{q}}}\,d\tau\right)^{\frac{2}{\widehat{r}}\left(1+\kappa\right)}
\end{aligned}
\end{equation}
where $A_{l,r}^-(t)=\left\{x\in B_r:\left(u-l\right)_->0\right\}$.\\
\indent To figure out the difficulties from the diffusion coefficients $U^{m-1}$, we consider the function $u_{\omega}=\max\left\{u,\frac{\omega}{4}\right\}$ which is introduced in \cite{HU}. Then
\begin{equation}\label{eq-controlling-the-second-opart-gradient-derivatives-to-absolutely-positive}
\begin{aligned}
-II&\geq\frac{mC_1\lambda\theta_0^{\,\alpha_0}}{2}\int_{-\theta_0^{-\alpha_{0}}R_i^2}^{t}\int_{B_{R_i}}\left|\nabla\left(u_{\omega}-l_i\right)_-\right|^2\eta_i^2\,dxd\tau\\
&\qquad \qquad-\frac{mC_2^22^{2i+2}\Lambda^{m-1}\left(2\theta_0\right)^2}{C_1R_i^2}\int_{-\theta_0^{-\alpha_{0}}R_i^2}^{t}\int_{B_{R_i}}\chi_{\left[u\leq l_i\right]}\,dxd\tau\\
&\qquad \qquad \qquad \qquad -2m\int_{-\theta_0^{-\alpha_{0}}R_i^2}^{t}\int_{B_{R_i}}U^{m-1}u^2\chi_{_{\left[u\leq l_j\right]}}\left(f_1+f_2\right)\left(\eta_i^2+\eta_i\left|\eta_i\right|^2+\eta_i\right)\,dxd\tau
\end{aligned}
\end{equation}
By \eqref{eq-energy-type-equation-with-Lebesgue-Steklov-everage}, \eqref{eq-controlling-the-first-opart-time-derivatives-to-sup}, \eqref{eq-estimates-for-III-forcing-term-2}, \eqref{eq-controlling-the-third-opart-gradient-derivatives-to-v-0}, \eqref{eq-controlling-the-third-opart-gradient-derivatives-to-v} and \eqref{eq-controlling-the-second-opart-gradient-derivatives-to-absolutely-positive}, we get
\begin{equation}\label{eq-simplifying-of-eq-after-h-to-zero}
\begin{aligned}
&\sup_{-\theta_0^{-\alpha_{0}}R_i^2<t<0}\int_{B_{R_i}\times\left\{t\right\}}\left(u_{\omega}-l_i\right)_-^2\eta_i^2\,dx+\theta_0^{\alpha_{0}}\int_{-\theta_0^{-\alpha_{0}}R_i^2}^{0}\int_{B_{R_i}}\left|\nabla\left(u_{\omega}-l_i\right)_-\eta_i\right|^2\,dxdt\\
&\qquad \qquad \leq C_1\theta_0^2\left[\frac{2^{2\left(i+1\right)}\left(\theta_0^{\alpha_0}+\Lambda^{m-1}\right)}{R_i^2}\int_{-\theta_0^{-\alpha_0}R_i^2}^{0}\int_{B_{R_i}}\chi_{\left[u_{\omega}\leq l_i\right]}\,dxdt+\frac{\Lambda^{m-1}+\Lambda^{2(b-1)}}{\theta_0^{\alpha_0}}\left(\int_{-\theta^{-\alpha_{0}}R_j^2}^{t}\left|A_{l_j,R_j}^-(\tau)\right|^{\frac{\widehat{r}}{\widehat{q}}}\,d\tau\right)^{\frac{2}{\widehat{r}}\left(1+\kappa\right)}\right]
\end{aligned}
\end{equation}
for some constant $C_1$ depending on $m$ and $\lambda$. We now  take the change of variables
\begin{equation*}
z=\theta_0^{\alpha_{0}}\,t
\end{equation*}
and set the new functions
\begin{equation*}
\overline{u}_{\omega}\left(\cdot,z\right)=u_{\omega}\left(\cdot,\theta_0^{-\alpha_{0}}z\right) \qquad \mbox{and} \qquad \overline{\eta}_{i}\left(\cdot,z\right)=\eta_{i}\left(\cdot,\theta_0^{-\alpha_{0}}z\right).
\end{equation*}
Then, by \eqref{eq-simplifying-of-eq-after-h-to-zero}
\begin{equation}\label{eq-simplifying-with-change-of-variables}
\begin{aligned}
&\sup_{-R_i^2<z<0}\int_{B_{R_i}\times\left\{z\right\}}\left(\overline{u}_{\omega}-l_i\right)_-^2\overline{\eta}_i^2\,dx+\int_{-R_i^2}^{0}\int_{B_{R_i}}\left|\nabla\left(\overline{u}_{\omega}-l_i\right)_-\overline{\eta}_i\right|^2\,dxdz \\
&\qquad \qquad \qquad \leq  C_1\theta_0^2\left[\frac{2^{2\left(i+1\right)}}{R_i^2}\left(1+\left(\frac{\Lambda}{\theta_0^{\,\beta}}\right)^{m-1}\right)A_i+\left(\Lambda^{m-1}+\Lambda^{2(b-1)}\right)\theta^{-\alpha_0\left(2-\frac{1}{q_2}\right)}\left(\int_{-R_j^2}^{0}\left|A_{i}(z)\right|^{\frac{\widehat{r}}{\widehat{q}}}\,dz\right)^{\frac{2}{\widehat{r}}\left(1+\kappa\right)}\right]
\end{aligned}
\end{equation}
where 
\begin{equation*}
A_i=\int_{-R_i^2}^{0}\int_{B_{R_i}}\chi_{\left[\overline{u}_{\omega}\leq l_i\right]}\,dxdz \qquad \mbox{and}\qquad A_i(z)=\left\{x\in B_{R_j}:\overline{u}_{_{\omega}}(x,z)<l_i\right\}.
\end{equation*}
By Lemma \ref{eq-Sobolev-Type-Inequality} and \eqref{eq-simplifying-with-change-of-variables},
\begin{equation}\label{eq-simplifying-after-change-of-variables-with-A-i}
\begin{aligned}
&\left\|\left(\overline{u}_{\omega}-l_i\right)_-^2\overline{\eta}_i^2\right\|_{L^2\left(Q\left(R_i,R_i^2\right)\right)} \\
&\qquad \qquad \leq C\theta_0^2\,A_i^{\frac{2}{n+2}}\left[\frac{2^{2\left(i+1\right)}}{R_i^2}\left(1+\left(\frac{\Lambda}{\theta_0^{\,\beta}}\right)^{m-1}\right)A_i+\left(\Lambda^{m-1}+\Lambda^{2(b-1)}\right)\theta^{-\alpha_0\left(2-\frac{1}{q_2}\right)}\left(\int_{-R_j^2}^{0}\left|A_{i}(z)\right|^{\frac{\widehat{r}}{\widehat{q}}}\,dz\right)^{\frac{2}{\widehat{r}}\left(1+\kappa\right)}\right].
\end{aligned}
\end{equation}
Choose the number $\epsilon>0$ sufficiently small that
\begin{equation}\label{eq-condition-of-epsilon-for-holder-continuity-at-start}
\epsilon<\frac{nq_2\kappa}{2q_2-1}. 
\end{equation}
Then by an argument similar to the proof of Lemma 3.5 of \cite{KL1}, there exists a constant $C_2>0$ depending on $m$, $\lambda$ and $\frac{\Lambda}{\theta^{\,\beta}}$ such that
\begin{equation}\label{eq-iteration-for-X-jk-and-Y-j-first}
X_{i+1}\leq C_216^i\left(X_i^{1+\frac{2}{n+2}}+\Lambda^{b-1}\theta^{-\alpha_0\left(2-\frac{1}{q_2}\right)}R_i^{n\kappa}X_i^{\frac{2}{n+2}}Y_i\right)\leq C16^i\left(X_i^{1+\frac{2}{n+2}}+X_i^{\frac{2}{n+2}}Y_i\right) \qquad \forall i\in\N
\end{equation}
and
\begin{equation}\label{eq-iteration-for-X-jk-and-Y-j-second}
Y_{i+1}\leq C16^i\left(X_i+Y^{1+\kappa}_i\right)\qquad \forall i\in\N
\end{equation}
where
\begin{equation*}
X_i=\frac{A_i}{\left|Q\left(R_i,R_i^2\right)\right|} \qquad \mbox{and}\qquad Y_i=\frac{1}{\left|B_{R_i}\right|}\left(\int_{-R_i^2}^{0}\left|A_{i}(z)\right|^{\frac{\widehat{r}}{\widehat{q}}}\,dz\right)^{\frac{2}{\widehat{r}}}.
\end{equation*}
By \eqref{eq-iteration-for-X-jk-and-Y-j-first} and \eqref{eq-iteration-for-X-jk-and-Y-j-second}, there exist a constant $C>0$ such that
\begin{equation*}
L_{i+1}\leq C16^{i(1+\kappa)}L_i^{1+\widehat{\kappa}}\qquad \forall j\in\N
\end{equation*}
where $L_i=X_i+Y_i^{1+\kappa}$ and $\widehat{\kappa}=\min\left\{\kappa,\frac{2}{n+2}\right\}$. If we take the constant $\rho_0>0$ in \eqref{eq-first-condition-of-small-region-of-lower-for-holder-estimates} sufficiently small that
\begin{equation*}
L_0\leq C^{-\frac{1+\kappa}{\widehat{\kappa}}}16^{-\frac{1+\kappa}{\widehat{\kappa}^2}}
\end{equation*}
holds, then
\begin{equation*}
L_i\leq C^{-\frac{(1+\kappa)(1+\widehat{\kappa})}{\widehat{\kappa}}}16^{-\frac{(1+\kappa)(1+i\,\widehat{\kappa})}{\widehat{\kappa}^2}} \to 0 \qquad \mbox{as $i\to\infty$}
\end{equation*}
and the lemma follows.
\end{proof}

\begin{remark}
If $U$ is equivalent to $u^{\beta}$, i.e., there exists some constants $0<c\leq C<\infty$ such that
\begin{equation*}
cu^{\beta}\leq U\leq Cu^{\beta} \qquad \mbox{in $Q\left(R,\theta_0^{-\alpha_0}R^{2}\right)$},
\end{equation*}
then the constant $\rho_0$ in \eqref{eq-first-condition-of-small-region-of-lower-for-holder-estimates} is independent of $U$ and $\omega$.
\end{remark}
\begin{remark}\label{remark-first-alternative-for-fast-diffusion-type-system}
For a constant $0<a<1$, let's take $au^{a-1}\left(u^a-l_i^a\right)_-\eta_i^2$ as a test function in the proof of Lemma  \ref{lem-the-first-alternative-for-holder-estimates}. Then, we can get the following energy type inequality
\begin{equation}\label{eq-energy-type-inequality-for-fast-diffysion-type-system}
\begin{aligned}
&\sup_{-\theta^{\,-\alpha_{0}}R_i^2<t<0}\int_{B_{R_j}}\left(u^{\,a}-l^{\,a}_j\right)_-^2\eta_j^2\,dx+\int_{-\theta^{-\alpha_{0}}R_j^2}^{0}\int_{B_{R_j}}U^{\,m-1}\left|\nabla\left(u^{\,a}-l_j\right)_-\eta_j\right|^2\,dxdt\\
&\qquad \qquad \leq C\theta_0^{2a}\Bigg[\frac{2^{2\left(i+1\right)}\left(\theta_0^{\alpha_0}+\Lambda^{m-1}\left(\frac{\Lambda}{\theta_0^{\alpha_0}}\right)^{a}\right)}{R_i^2}\int_{-\theta_0^{-\alpha_0}R_i^2}^{0}\int_{B_{R_i}}\chi_{\left[u\leq l_i\right]}\,dxdt\\
&\qquad \qquad \qquad \qquad\qquad \qquad  +\frac{\Lambda^{m-1}\left(\frac{\Lambda}{\theta_0^{\alpha_0}}\right)^{a}+\Lambda^{2(b-1)}}{\theta_0^{\alpha_0}}\left(\int_{-\theta^{-\alpha_{0}}R_j^2}^{t}\left|A_{l_j,R_j}^-(\tau)\right|^{\frac{\widehat{r}}{\widehat{q}}}\,d\tau\right)^{\frac{2}{\widehat{r}}\left(1+\kappa\right)}\Bigg]
\end{aligned}
\end{equation}
If we choose the constant $a$ satisfying
\begin{equation*} 
m-1+a>0,
\end{equation*} 
then the right hand side of \eqref{eq-energy-type-inequality-for-fast-diffysion-type-system} will be controllable by the measure of the set $\left\{u\leq l_i\right\}$ even though $0<m<1$. Therefore, by similar arguments the lemma is still true for $0<m<1$, i.e., the first alternative can be extend to the fast diffusion type system.
\end{remark}

\subsection{The Second Alternative}

Suppose that the assumption of Lemma \ref{lem-the-first-alternative-for-holder-estimates} does not hold, i.e., for every sub-cylinder $Q\left(R,\theta_0^{-\alpha_{0}}R^2\right)$
\begin{equation*}\label{eq-start-point-of-the-second-alternative-assumption-violated}
\left|\left\{\left(x,t\right)\in Q\left(R,\theta_0^{-\alpha_{0}}R^2\right):u(x,t)<\frac{\omega}{2}\right\}\right|>\rho_0\left|Q\left(R,\theta_0^{-\alpha_{0}}R^2\right)\right|.
\end{equation*}
Then
\begin{equation*}
\left|\left\{\left(x,t\right)\in Q\left(R,\theta_0^{-\alpha_{0}}R^2\right):u(x,t)>\frac{\omega}{2}\right\}\right|\leq\left(1-\rho_0\right)\left|Q\left(R,\theta_0^{-\alpha_{0}}R^2\right)\right|
\end{equation*}
is valid for all cylinders 
\begin{equation*}
Q\left(R,\theta_0^{-\alpha_{0}}R^2\right)\subset Q\left(R, R^{2-\epsilon}\right).
\end{equation*}
By an argument similar to the Lemma 4.2 of \cite{KL2}, we have the following lemma

\begin{lemma}\label{lem-propostion-of-area-near-supremum-at-some-time-t-ast}
If \eqref{eq-first-condition-of-small-region-of-lower-for-holder-estimates} is violated, then there exists a time level
\begin{equation*}
t^{\ast}\in\left[-\theta_0^{-\alpha_{0}}R^2,-\frac{\rho_0}{2}\theta_0^{-\alpha_{0}}R^2\right]
\end{equation*}
such that
\begin{equation*}
\left|\mathcal{A}_0\right|=\left|\left\{x\in B_{R}:u\left(x,t^{\ast}\right)>\frac{\omega}{2}\right\}\right|<\left(\frac{1-\rho_0}{1-\frac{\rho_0}{2}}\right)\left|B_R\right|.
\end{equation*}
\end{lemma}
By Lemma \ref{lem-propostion-of-area-near-supremum-at-some-time-t-ast}, there exists a time $t^{\ast}<0$ such that the region $\mathcal{A}_0$ takes a portion of the ball $B_R$. The next lemma shows that this occurs for all $t\geq t^{\ast}$. 
\begin{lemma}
There exists a positive integer $s_1>1$ depending on $\frac{\Lambda}{\theta_0^{\,\beta}}$ such that
\begin{equation}\label{eq-supremum-small-occurs-for-all-time-near-top}
\left|\left\{x\in B_{R}:u(x,t)>\left(1-\frac{1}{2^{s_1}}\right)\omega\right\}\right|<\left(1-\left(\frac{\rho_0}{2}\right)^2\right)\left|B_R\right|,\qquad \forall t\in\left[t^{\ast},0\right].
\end{equation}
\end{lemma}

\begin{proof}
We will use a modification of the proof of Lemma 3.7 of \cite{KL1} to prove the lemma. Let 
\begin{equation*}
H=\sup_{B_{R}\times\left[t^{\ast},0\right]}\left(u-\left(\mu^+-\frac{\omega}{2}\right)\right)_+\leq \frac{\omega}{2}
\end{equation*}
and assume that there exists a constant $1<s_2\in\N$ such that
\begin{equation*}
0<\frac{\omega}{2^{s_2+1}}<H.
\end{equation*}
If there's no such integer $s_2$, \eqref{eq-supremum-small-occurs-for-all-time-near-top} holds for any $s_1>1$ and the lemma follows.\\
\indent We now introduce the logarithmic function which appears in Section 2 of \cite{Di} by
\begin{equation*}
\Psi\left(H,\left(u-k\right)_+,c\right)=\max\left\{0,\,\, \log\left(\frac{H}{H-\left(u-k\right)_++c}\right)\right\}
\end{equation*}
for $k=\frac{\omega}{2}$ and $c=\frac{\omega}{2^{s_2+1}}$. Note that
\begin{equation}\label{eq-region-where-Psi-zero-on-u-leq-k}
\Psi\left(H,\left(u-k\right)_+,c\right)=0 \qquad \mbox{if $u\leq k=\frac{\omega}{2}$}.
\end{equation}
For simplicity, we let $\psi\left(u\right)=\Psi\left(H,\left(u-k\right)_+,c\right)$. Then $\psi$ satisfies
\begin{equation}\label{eq-condition-of-psi-from-Psi-H-u-k-c}
\psi\leq s_2\log 2, \qquad 0\leq \left(\psi\right)'\leq \frac{2^{s_2+1}}{\omega} \qquad \mbox{and} \qquad \psi''=\left(\psi'\right)^2\geq 0.
\end{equation}
Set 
\begin{equation*}
\vp=\left(\psi^2\left(u_h\right)\right)'\xi^2
\end{equation*} 
and take it as a test function in \eqref{eq-formulation-for-weak-solution-of-u-h} where $u_h$ is the Lebesgue-Steklov average of $u$ and $\xi(x)\geq 0$ is a smooth cut-off function such that
\begin{equation}\label{eq-condition-of-cut-off-function-independent-of-t}
\xi=1 \quad \mbox{in $B_{\left(1-\nu\right)R}$}, \qquad \xi=0 \quad \mbox{on $\partial B_{R}$} \qquad \mbox{and} \qquad \left|\nabla\xi\right|\leq \frac{C}{\nu R} 
\end{equation}
for some constants $0<\nu<1$ and $C>0$. Then integrating \eqref{eq-formulation-for-weak-solution-of-u-h} over $\left(t^{\ast},t\right)$ for all $t\in\left(t^{\ast},0\right)$, we have
\begin{equation}\label{eq-first-step-of-multiplying-test-=function-for-second-alternativwes}
\begin{aligned}
0&=\int_{t^{\ast}}^t\int_{B_R}\left(\psi^2\left(u_h\right)\xi^2\right)_{\tau}\,dxd\tau+m\int_{t^{\ast}}^t\int_{B_R}\left(U^{m-1}\mathcal{A}\left(\nabla u,u,x,t\right)\right)_h\cdot\nabla\left(\left(\psi^2\left(u_h\right)\right)'\xi^2\right)\,dxd\tau\\
&\qquad \qquad +\int_{t^{\ast}}^t\int_{B_R}\left(\mathcal{B}\left(u,x,t\right)\right)_h\cdot\nabla\left(\left(\psi^2\left(u_h\right)\right)'\xi^2\right)\,dxd\tau\\ 
&=I+II+III.
\end{aligned}
\end{equation}
Then we have
\begin{align}
I\to\int_{B_R\times\left\{t\right\}}\psi^2\left(u\right)\xi^2\,dx-\int_{B_R\times\left\{t^{\ast}\right\}}\psi^2\left(u\right)\xi^2\,dx \qquad \qquad \mbox{as $h\to 0$}\label{eq-first-integral-removing-time-derivatives-1}
\end{align}
and
\begin{equation}\label{eq-second-integral-using-youngs-indequality-after-h-to-010}
\begin{aligned}
II&\to m\int_{t^{\ast}}^t\int_{B_R}U^{m-1}\mathcal{A}\left(\nabla u,u,x,t\right)\cdot\nabla\left(\left(\psi^2\left(u\right)\right)'\xi^2\right)\,dxd\tau \qquad \mbox{as $h\to 0$}\\
&\geq C_1m\int_{t^{\ast}}^t\int_{B_R}U^{m-1}\left(1+\psi\right)\left(\psi'\right)^2\xi^2\left|\nabla u\right|^2\,dxd\tau\\
&\qquad \qquad -2m\int_{t^{\ast}}^t\int_{B_R}U^{m-1}\left(1+\psi\right)\left(\psi'\right)^2\xi^2u^2f_1\,dxd\tau\\
&\qquad \qquad  \qquad \qquad -8C_1m\int_{t^{\ast}}^t\int_{B_R}U^{m-1}\psi\left|\nabla \xi\right|^2\,dxd\tau \\
&\qquad \qquad  \qquad \qquad \qquad \qquad -4C_1m\int_{t^{\ast}}^t\int_{B_R}U^{m-1}\psi\left(\psi'\right)^2\xi^2u^2f_2^2\,dxd\tau
\end{aligned}
\end{equation}
and
\begin{equation}\label{eq-third-integral-using-youngs-indequality-after-h-to-010}
\begin{aligned}
-III\to &\int_{t^{\ast}}^t\int_{B_R}\mathcal{B}\left(u,x,t\right)\cdot\nabla\left(\left(\psi^2\left(u\right)\right)'\xi^2\right)\,dxd\tau \qquad \mbox{as $h\to 0$}\\
&\qquad =2\int_{t^{\ast}}^t\int_{B_R}u^b\left(1+\psi\right)\left(\psi'\right)^2\xi^2f_3\left|\nabla u\right|\,dxd\tau+4\int_{t^{\ast}}^t\int_{B_R}u^b\psi\psi'\xi f_3\left|\nabla \xi\right|\,dxd\tau\\
&\qquad \leq C_1m\lambda\theta_0^{\,\alpha_0}\int_{t^{\ast}}^t\int_{B_R}\left(1+\psi\right)\left(\psi'\right)^2\xi^2\left|\nabla u\right|^2\,dxd\tau\\
&\qquad \qquad+\frac{4(1+m)}{C_1m\lambda\theta_0^{\,\alpha_0}}\int_{t^{\ast}}^t\int_{B_R}u^{2b}\left(1+\psi\right)\left(\psi'\right)^2\xi^2f_3^2\,dxd\tau\\
&\qquad \qquad \qquad \qquad +2C_1\Lambda^{m-1}\int_{t^{\ast}}^t\int_{B_R}\psi\left|\nabla \xi\right|^2\,dxd\tau.
\end{aligned}
\end{equation}
By Lemma \ref{lem-propostion-of-area-near-supremum-at-some-time-t-ast}, \eqref{eq-condition-of-psi-from-Psi-H-u-k-c}, \eqref{eq-condition-of-cut-off-function-independent-of-t}, \eqref{eq-first-step-of-multiplying-test-=function-for-second-alternativwes}, \eqref{eq-first-integral-removing-time-derivatives-1}, \eqref{eq-second-integral-using-youngs-indequality-after-h-to-010} and \eqref{eq-third-integral-using-youngs-indequality-after-h-to-010},
\begin{equation}\label{eq-after-togethering-the-frist-and-second-integral0-and-simplify}
\begin{aligned}
\int_{B_R\times\left\{t\right\}}\psi^2\left(u\right)\xi^2\,dx\leq \left[s_2^2\left(\log2\right)^2\left(\frac{1-\rho_0}{1-\frac{\rho_0}{2}}\right)+C\left(\frac{s_2\log 2}{\nu^2}\left(\frac{\Lambda}{\theta_0^{\,\beta}}\right)^{m-1}+\left(\Lambda^{m+1}+\Lambda^{2(b-1)}\right)4^{s_2+1}R^{2-2\epsilon}s_2\log 2\right)\right]\left|B_R\right|
\end{aligned}
\end{equation}
holds for all $t\in\left(t^{\ast},0\right)$ with some constant $C>0$ depending $m$, $b$, $\lambda$ and $\left\|f_i\right\|_{L^{\infty}}$, $\left(i=1,2,3\right)$. Let 
\begin{equation*}
\mathcal{S}=\left\{x\in B_{(1-\nu)R}:u(x,t)>\left(1-\frac{1}{2^{s_2+1}}\right)\omega\right\}.
\end{equation*}
Then the left hand side of \eqref{eq-after-togethering-the-frist-and-second-integral0-and-simplify} is bounded below by
\begin{equation}\label{lower-bound-of-the-left-hand0sidelintegarl-of-simplify}
\int_{B_R\times\left\{t\right\}}\psi^2\left(u\right)\xi^2\,dx\geq \int_{\mathcal{S}}\psi^2\left(u\right)\xi^2\,dx\geq\left(s_2-1\right)^2\left(\log 2\right)^2\left|\mathcal{S}\right| \qquad \forall t\in\left(t^{\ast},0\right). 
\end{equation}
Observe that
\begin{equation}\label{eq-volume-of-set-greater-than-mu--2-to-=s-0-and-s-2-with-mathcal-S}
\left|\left\{x\in B_{R}:u(x,t)>\left(1-\frac{1}{2^{s_2+1}}\right)\omega\right\}\right|\leq\left|\mathcal{S}\right|+N\nu\left|B_R\right|.
\end{equation}
By \eqref{eq-after-togethering-the-frist-and-second-integral0-and-simplify}, \eqref{lower-bound-of-the-left-hand0sidelintegarl-of-simplify} and \eqref{eq-volume-of-set-greater-than-mu--2-to-=s-0-and-s-2-with-mathcal-S},
\begin{equation*}
\begin{aligned}
&\left|\left\{x\in B_{R}:u(x,t)>\left(1-\frac{1}{2^{s_2+1}}\right)\omega\right\}\right|\\
&\qquad \qquad \leq \left[\left(\frac{s_2}{s_2-1}\right)^2\left(\frac{1-\rho_0}{1-\frac{\rho_0}{2}}\right)+N\nu+C\left(\frac{s_2}{\nu^2(s_2-1)^2\log 2}\left(\frac{\Lambda}{\theta^{\,\beta}}\right)^{m-1}+\frac{\left(\Lambda^{m+1}+\Lambda^{2(b-1)}\right)4^{s_2+1}R^{2-2\epsilon}s_2}{\left(s_2-1\right)^2\log 2}\right)\right]\left|B_R\right|.
\end{aligned}
\end{equation*}
To complete the proof, we choose $\nu$ so small that $n\nu\leq \frac{3}{8}\rho_0^2$ and then $s_2$ so large that
\begin{equation*}
\left(\frac{s_2}{s_2-1}\right)^2\leq \left(1-\frac{1}{2}\rho_0\right)\left(1+\rho_0\right) \qquad \mbox{and} \qquad C\frac{s_2}{\nu^2(s_2-1)^2\log 2}\left(\frac{\Lambda}{\theta^{\,\beta}}\right)^{m-1}\leq \frac{1}{4}\rho_0^2.
\end{equation*}
With such $\nu$ and $s_2$, we choose the radius $R$ sufficiently small that
\begin{equation}\label{eq-range-of-R-sufficiently-small-second-2}
\frac{C\left(\Lambda^{m+1}+\Lambda^{2(b-1)}\right)4^{s_2+1}R^{2-2\epsilon}s_2}{\left(s_2-1\right)^2\log 2}\leq \frac{3}{8}\rho_0^2.
\end{equation}
Then \eqref{eq-supremum-small-occurs-for-all-time-near-top} holds for $s_1=s_2+1$ and the lemma follows.
\end{proof}

Since $t^{\ast}\in\left[-\theta_0^{-\alpha_{0}}R^2,-\frac{\rho_0}{2}\theta_0^{-\alpha_{0}}R^2\right]$, the previous lemma implies the following result.
\begin{cor}\label{cor--propostion-of-area-near-supremum-at-some-time-t-ast}
There exists a positive integer $s_1>s_0$ such that for all $t\in\left(-\frac{\rho_0}{2}\theta_0^{-\alpha_{0}}R^2,0\right)$
\begin{equation}\label{eq-supremum-small-occurs-for-all-time-near-top-2}
\left|\left\{x\in B_{R}:u(x,t)>\left(1-\frac{1}{2^{s_1}}\right)\omega\right\}\right|<\left(1-\left(\frac{\rho_0}{2}\right)^2\right)\left|B_R\right|.
\end{equation}
\end{cor}

To make the region where $u$ is close to its supremum to be arbitrary small, we review the following lemma.
\begin{lemma}[De Giorgi\cite{De}]\label{De-giorgi}
If $f\in W^{1,1}(B_r)$ $(B_r\subset\R^n)$ and $l,k\in \R$, $k<l$,
then
\begin{equation*}
(l-k)\left|\left\{x\in B_r: f(x)>l\right\}\right|\leq
\frac{Cr^{n+1}}{\left|\left\{x\in B_r:
f(x)<k\right\}\right|}\int_{k<f<l}|\nabla f|\,\,dx,
\end{equation*}
where $C$ depends only on $n$.
\end{lemma}
By Corollary \ref{cor--propostion-of-area-near-supremum-at-some-time-t-ast} and Lemma \ref{De-giorgi}, we have the following lemma.
\begin{lemma}\label{lem-last-condition-for-second-alternativ-from-violate-of-first-alternative}
If \eqref{eq-first-condition-of-small-region-of-lower-for-holder-estimates} is violated, for every $\nu_{\ast}\in\left(0,1\right)$ there exists a natural number $s^{\ast}>s_1>1$ depending on $\frac{\Lambda}{\theta^{\,\beta}}$ such that
\begin{equation}\label{eq-condition-with-s-upper-ast-and-nu--sub-ast-for-second-alternative}
\left|\left\{(x,t)\in Q\left(R,\frac{\rho_0}{2}\theta_0^{-\alpha_{0}}R^2\right): u(x,t)>\left(1-\frac{1}{2^{s^{\ast}}}\right)\omega\right\}\right|\leq \nu_{\ast}\left|Q\left(R,\frac{\rho_0}{2}\theta_0^{-\alpha_{0}}R^2\right)\right|.
\end{equation}
\end{lemma}

\begin{proof}
We will use a  modification of the proof of Lemma 8.1 of Section III of \cite{Di} to prove the lemma. Let $k=\left(1-\frac{1}{2^s}\right)\omega$ for $s\geq s_1$ and let $\eta(x,t)\in C^{\infty}\left(Q\left(2R,\rho_0\theta^{-\alpha_{0}}R^2\right)\right)$ be a cut-off function such that
\begin{equation*}
\begin{cases}
\begin{array}{cccl}
0\leq \eta\leq 1 &&& \mbox{in $Q\left(2R,\rho_0\theta^{-\alpha_{0}}R^2\right)$}\\
\eta=1 &&& \mbox{in $Q\left(R,\frac{\rho_0}{2}\theta^{-\alpha_{0}}R^2\right)$ }\\
\eta=0 &&& \mbox{on the parabolic boundary of $Q\left(2R,\rho_0\theta^{-\alpha_{0}}R^2\right)$}\\
\left|\nabla\eta\right|\leq \frac{1}{R},\qquad \left|\eta_t\right|\leq \frac{2\theta^{\alpha_{0}}}{\rho_0R^2}.&&&
\end{array}
\end{cases}
\end{equation*} 
Put $\vp=\left(u_h-k\right)_+\xi^2$ in the weak formula \eqref{eq-formulation-for-weak-solution-of-u-h}, Integrate it over $\left(-\rho_0\theta^{-\alpha_{0}}R^2,t\right)$ for $t\in\left(-\rho_0\theta^{-\alpha_{0}}R^2,0\right)$ and take the limit as $h\to 0$. Then, by an argument simlar to the proof of Energy type inequality \eqref{eq-simplifying-of-eq-after-h-to-zero}  we have 
\begin{equation}\label{eq-upper-bound-of-gradient-of-u-i-minus-k-in-second-alternative}
\begin{aligned}
&\int_{-\frac{\rho_0}{2}\theta^{-\alpha_{0}}R^2}^t\int_{B_{R}}\left|\nabla \left(u^{\,i}-k\right)_+\right|^2\,dx\,d\tau\\
&\qquad \leq C\left(\frac{\omega}{2^s}\right)^2\frac{1}{R^2}\left(1+\left(\frac{\Lambda}{\theta^{\,\beta}}\right)^{m-1}+2^sR^{n\kappa-\epsilon\left(2-\frac{1}{q_2}\right)}\right)\left|Q\left(R,\frac{\rho_0}{2}\theta^{-\alpha_{0}}R^2\right)\right|
\end{aligned}
\end{equation}
for some constant $C>0$. Let
\begin{equation*}
A_s\left(t\right)=\left\{x\in B_{R}:u(x,t)>\left(1-\frac{1}{2^s}\right)\omega\right\}, \qquad \forall t\in\left(-\frac{\rho_0}{2}\theta^{-\alpha_{0}}R^2,0\right)
\end{equation*} 
and 
\begin{equation*}
A_s=\int_{-\frac{\rho_0}{2}\theta^{-\alpha_{0}}R^2}^0\left|A_s(t)\right|\,dt.
\end{equation*}
Then, by Corollary \ref{cor--propostion-of-area-near-supremum-at-some-time-t-ast}, Lemma \ref{De-giorgi} and \eqref{eq-upper-bound-of-gradient-of-u-i-minus-k-in-second-alternative} we have
\begin{equation*}
\begin{aligned}
&\left(\frac{\omega}{2^{s+1}}\right)\left|A_{s+1}(t)\right|\leq \frac{CR}{\rho_0^2}\int_{\left\{\left(1-\frac{1}{2^s}\right)\omega<u<\left(1-\frac{1}{2^{s+1}}\right)\omega\right\}}\left|\nabla u\right|\,dx\qquad \qquad\forall s=s_1,\cdots,s^{\ast}-1\\
&\Rightarrow \qquad \left(\frac{\omega_{_M}}{2^{s+1}}\right)A_{s+1}\leq \frac{CR}{\rho_0^2}\left(\int_{-\frac{\rho_0}{2}\theta^{-\alpha_{0}}R^2}^0\int_{B_R}\left|\nabla(u-k)_+\right|^2\,dx\,dt\right)^{\frac{1}{2}}\left|A_s\bs A_{s+1}\right|^{\frac{1}{2}}\\
&\Rightarrow \qquad A_{s+1}^2\leq \frac{C}{\rho_0^4}\left(1+\left(\frac{\Lambda}{\theta^{\,\beta}}\right)^{m-1}+2^{s^{\ast}}R^{n\kappa-\epsilon\left(2-\frac{1}{q_2}\right)}\right)\left|Q\left(R,\frac{\rho_0}{2}\theta^{-\alpha_{0}}R^2\right)\right|\left|A_s\bs A_{s+1}\right|  \\
&\Rightarrow \qquad \left(s^{\ast}-s_1\right)A_{s^{\ast}}^2\leq \sum_{s=s_1}^{s^{\ast}-1}A_{s+1}^2\leq \frac{C}{\rho_0^4}\left(1+\left(\frac{\Lambda}{\theta^{\,\beta}}\right)^{m-1}+2^{s^{\ast}}R^{n\kappa-\epsilon\left(2-\frac{1}{q_2}\right)}\right)\left|Q\left(R,\frac{\rho_0}{2}\theta^{-\alpha_{0}}R^2\right)\right|\left|A_{s_1}\bs A_{s^{\ast}}\right|\\
&\Rightarrow \qquad A_{s^{\ast}}^2\leq \frac{C}{\rho_0^4\left(s^{\ast}-s_1\right)}\left(1+\left(\frac{\Lambda}{\theta^{\,\beta}}\right)^{m-1}+2^{s^{\ast}}R^{n\kappa-\epsilon\left(2-\frac{1}{q_2}\right)}\right)\left|Q\left(R,\frac{\rho_0}{2}\theta^{-\alpha_{0}}R^2\right)\right|^2.
\end{aligned}
\end{equation*}
Thus if we choose $s^{\ast}\in\N$ sufficiently large that
\begin{equation*}
\frac{C}{\rho_0^4\left(s^{\ast}-s_1\right)}\left(2+\left(\frac{\Lambda}{\theta^{\,\beta}}\right)^{m-1}\right)\leq \nu_{\ast}^2
\end{equation*}
and then $R$ sufficiently small that
\begin{equation}\label{eq-second-condition-for-range-of-R-by-s-ast}
2^{2s^{\ast}}R^{n\kappa-\epsilon\left(2-\frac{1}{q_2}\right)}\leq 1,
\end{equation}
then \eqref{eq-condition-with-s-upper-ast-and-nu--sub-ast-for-second-alternative} holds and the lemma follows.
\end{proof}

\begin{remark}
If $U$ is equivalent to $u^{\beta}$, i.e., there exists some constants $0<c\leq C<\infty$ such that
\begin{equation*}
cu^{\beta}\leq U\leq Cu^{\beta} \qquad \mbox{in $Q\left(R,\theta_0^{-\alpha_0}R^{2}\right)$},
\end{equation*}
then the constant $s^{\ast}$ is independent of $U$ and $\omega$.
\end{remark}
By Lemma \ref{lem-last-condition-for-second-alternativ-from-violate-of-first-alternative}, we have a similar assumption to the one in Lemma \ref{lem-the-first-alternative-for-holder-estimates} for sufficiently small number $\nu_{\ast}>0$. Therefore, by an argument similar to the proof of Lemma \ref{lem-the-first-alternative-for-holder-estimates}, we can have the following result.

\begin{lemma}\label{lem-the-second-alternative-for-holder-estimates}
The number $\nu_{\ast}\in\left(0,1\right)$ can be chosen such that
\begin{equation*}
u(x,t)\leq \left(1-\frac{1}{2^{s^{\ast}+1}}\right)\omega \qquad \mbox{a.e. on $Q\left(\frac{R}{2},\frac{\rho_0}{2}\theta_{0}^{-\alpha_{0}}\left(\frac{R}{2}\right)^2\right)$}. 
\end{equation*}
\end{lemma}
\begin{remark}\label{remark-explain-extension-to-fde-system-on-second-alternative}
Throughout the second alternative, the function $U$ satisfies
\begin{equation*}
0<\frac{\omega}{2}\leq U\leq \Lambda<\infty.
\end{equation*} 
Thus the diffusion coefficients $U^{m-1}$ will still be nondegenerate when $0<m<1$. Therefore the second alternative can be extended to the fast diffusion type system, i.e., the Lemma \ref{lem-the-second-alternative-for-holder-estimates} holds for $0<m<1$.
\end{remark}

\subsection{Local Continuity}

By Lemma \ref{lem-the-first-alternative-for-holder-estimates} and Lemma \ref{lem-the-second-alternative-for-holder-estimates}, we have the following Oscillation Lemma.
\begin{lemma}[Oscillation Lemma]\label{lem-Oscillation-Lemma} 
There exist numbers $\rho_0$, $\sigma_0\in\left(0,1\right)$ depending on the $\frac{\Lambda}{\theta_0^{\,\beta}}$ such that if
\begin{equation*}
\osc_{Q\left(R,\theta_0^{-\alpha_{0}}R^{2}\right)}u=\omega
\end{equation*}
then
\begin{equation}\label{eq-inequality-for-Oscillation-Lemma}
\osc_{Q\left(\frac{R}{2},\frac{\rho_0}{2}\theta_0^{-\alpha_{0}}\left(\frac{R}{2}\right)^{2}\right)}u=\sigma_0\omega.
\end{equation}
\end{lemma}
\begin{proof}[\textbf{Proof of Theorem \ref{eq-local-continuity-of-solution}}]
Let $\left\{\omega_{n}\right\}$ be a  decreasing sequence such that
\begin{equation*}
\omega_{n}=\sigma_0^n\omega_0 \qquad \forall n\in\N.
\end{equation*}
By arguments similar to the proofs of Lemma \ref{lem-Oscillation-Lemma}, we can construct a family of nest and shrinking cylinders $\left\{Q_n\right\}_{n=1}^{\infty}$, whose radius is depending on $\sigma_0^n$, recursively such that 
\begin{equation}\label{eq-iterally-decreasing-of-oscillation-for-continuity-of-solution}
ess\sup_{Q_n}u\leq\omega_n.
\end{equation}
Thus, the continuity of $u$ follows. 
\end{proof}
\begin{remark}\label{remark-last-in-continuity-some-case-just-conti-under-some-condition-up-to-holder-estimate}
Under the \textbf{Assumption I and II}, the constant $\sigma_0$ in \eqref{eq-inequality-for-Oscillation-Lemma} may depend on the oscillation $\frac{\Lambda}{\theta_0^{\,\beta}}$. Thus we can only get the local continuity of $u$ and can't find the modulus of continuity at this stage. See \cite{Ur} for the details.
\end{remark}
By Oscillation Lemma, the shrinking rate of $Q_n$ and decay rate of $\omega_n$ depends only on $\frac{\Lambda}{\theta_0^{\,\beta}}$. Thus those rates are maintained for all $n\in\N$ if there exist some constants $0<c\leq C<\infty$ such that
\begin{equation*}
cu^{\,\beta}\leq U\leq Cu^{\,\beta} \qquad \mbox{in $Q\left(R,R^{2-\epsilon}\right)$}.
\end{equation*}
As a consequence of Theorem \ref{eq-local-continuity-of-solution}, we can find the modulus of continuity of solution $\bold{u}$ of \eqref{eq-pme_system} (H\"older regularity).
\begin{proof}[\textbf{Proof of Theorem \ref{cor-local-holder-estimates-of-u-i-s-in-solution-bold-u}}]
Since 
\begin{equation*}
U=u^1+\cdots+u^k\geq u^i\qquad \forall 1\leq i\leq k,
\end{equation*} 
$u^i$, $\left(1\leq i\leq k\right)$ satisfies the \textbf{Assumption I} with $\beta=1$, $\lambda=1$. By similar argument as in the first alternative, we can set
\begin{equation*}
ess\sup_{Q(R,R^{2-\epsilon})}u^i=\left(\mu^i\right)^+>0, \qquad ess\inf_{Q(R,R^{2-\epsilon})}u^i=\left(\mu^i\right)^-=0, \qquad \forall 1\leq i\leq k
\end{equation*}
and
\begin{equation*}
\omega^i=\osc_{Q(R,R^{2-\epsilon})} u^i=\left(\mu^i\right)^+-\left(\mu^i\right)^-=\left(\mu^i\right)^+,\qquad \forall 1\leq i\leq k.
\end{equation*}
Let 
\begin{equation*}
\omega_{_M}=\max_{1\leq i\leq k}\omega^i\qquad \mbox{and} \qquad \theta_0=\frac{\omega_{_M}}{4}.
\end{equation*}
Then
\begin{equation*}
\osc_{Q\left(R,\theta_0^{-\alpha_{0}}R^{2}\right)}u^i\leq \omega_{_M}=4\theta_0, \qquad \forall 1\leq i\leq k
\end{equation*}
and
\begin{equation*}
U=u^1+\cdots+u^k\leq k\omega_{_M}=4k\theta_0\qquad \mbox{on $Q(R,\theta_0^{-\alpha_0}R^{2})$}.
\end{equation*}
By the arguments similar to the proofs of \textbf{Oscillation Lemma} with $\Lambda$, $\mathcal{A}\left(\nabla u,u,x,t\right)$ and $\mathcal{B}\left(u,x,t\right)$ being replaced by $k\omega_{_M}=4k\theta_0$, $\nabla u$ and $0$, we can choose a constant $\sigma_0$ independent of $\omega^{1}$, $\cdots$, $\omega^k$ such that
\begin{equation*}
\osc_{Q\left(\frac{R}{2},\frac{\rho_0}{2}\theta_0^{-\alpha_{0}}\left(\frac{R}{2}\right)^{2}\right)}u^i=\sigma_0\omega_{_M} \qquad \forall 1\leq i\leq k.
\end{equation*}
Let 
\begin{equation*}
R_1=\frac{R}{\sigma_0^{\frac{\alpha_0}{2}}C}=\frac{R}{\overline{C}}
\end{equation*}
where $C$ is a constant satisfying
\begin{equation*}
C\geq \max\left(\frac{2}{\sqrt{\sigma_0^{\alpha_0}}}\,, \,2\sqrt{\frac{2}{\rho_0\sigma_0^{\alpha_0}}},\frac{1}{\sigma_0^{1+\frac{\alpha_0}{2}}}\right).
\end{equation*}
Then
\begin{equation*}
Q\left(R_1,\theta^{-\alpha_0}R_1^2\right) \subset Q\left(\frac{R}{C},\left(\sigma_0\theta\right)^{-\alpha_0}\left(\frac{R}{C}\right)^2\right) \subset Q\left(\frac{R}{2},\frac{\rho_0}{2}\theta^{-\alpha_0}\left(\frac{R}{2}\right)^2\right).
\end{equation*}
Thus we have
\begin{equation*}
\osc_{Q\left(R_1,\theta^{-\alpha_0}R_1^2\right)}u^i\leq \sigma_0\omega_{_M}, \qquad \forall 1\leq i\leq k.
\end{equation*}
and
\begin{equation*}
U=u^1+\cdots+u^k\leq k\sigma_0\omega_{_M}=4k\sigma_0\theta_0\qquad \mbox{on $Q(R_1,\theta_0^{-\alpha_0}R_1^{2})$}.
\end{equation*}
Thus, applying the \textbf{Oscillation Lemma} again with $\theta_0$, $\Lambda$ being replaced by $\sigma_{0}\theta_0$, $4k\sigma_0\theta_0$ respectively, we can have
\begin{equation}\label{eq-inequality-for-Holder-continuity-first-step}
\osc_{Q\left(\frac{R}{\overline{C}^2},\theta^{-\alpha_0}\left(\frac{R}{\overline{C}^2}\right)^2\right)}u^i=\osc_{Q\left(\frac{R_1}{\sigma_0^{\frac{\alpha_0}{2}}C},\left(\sigma_0^{2}\theta\right)^{-\alpha_0}\left(\frac{R_1}{C}\right)^2\right)}u^i\leq \osc_{Q\left(\frac{R_1}{2},\frac{\rho_0}{2}\left(\sigma_0\theta\right)^{-\alpha_{0}}\left(\frac{R_1}{2}\right)^{2}\right)}u^i=\sigma_0\omega_1=\sigma_0^2\omega_0, \qquad \forall 1\leq i\leq k.
\end{equation}
Continuing this process, we can have 
\begin{equation*}
\osc_{Q\left(\frac{R}{\overline{C}^j},\theta^{-\alpha_0}\left(\frac{R}{\overline{C}^j}\right)^2\right)}u^i\leq\sigma_0^j\omega_0 \qquad \forall j\in\N,\,\,1\leq i\leq k.
\end{equation*}
Therefore, by an argument similar to the proof of Theorem 3.12 of \cite{KL1}
\begin{equation*}
\osc_{Q\left(r,\frac{\rho_0}{2}\theta^{-\alpha_0}r^2\right)}u^i\leq K\omega_{_M}\left(\frac{r}{R}\right)^{\,\beta} \qquad \forall 0<r<R, 1\leq i\leq k.
\end{equation*}
holds for $0<\beta=-\log_{\overline{C}}\sigma_0<1$, $K=\frac{1}{\sigma_0}$ and the corollary follows.
\end{proof}
\begin{remark}
 By Remark \ref{remark-first-alternative-for-fast-diffusion-type-system} and Remark \ref{remark-explain-extension-to-fde-system-on-second-alternative}, the local continuity and local H\"older continuity can be extended to the fast diffusion type system, i.e., Theorem \ref{eq-local-continuity-of-solution} and Theorem \ref{cor-local-holder-estimates-of-u-i-s-in-solution-bold-u} holds for $0<m<1$.
\end{remark}

\section{Asymptotic Behaviour}
\setcounter{equation}{0}
\setcounter{thm}{0}

In this section, we will investigate the uniform convergence between the solution of \eqref{eq-for-each-population-u-i} and Barenblatt profile of porous medium equation. The self-similar Barenblatt solution of the porous medium equation with $L^1$-mass $M$ is given explicitly by
\begin{equation}\label{eq-barenblatt-solution-of-PME}
\mathcal{B}_M(x,t)=t^{-a_1}\left(\mathcal{C}_M-\frac{k|x|^2}{t^{2a_2}}\right)_+^{\frac{1}{m-1}}
\end{equation} 
where 
\begin{equation}\label{eq-constant-alpha-1-beta-1-k}
a_1=\frac{n}{(m-1)n+2}, \qquad a_2=\frac{a_1}{n},\qquad k=\frac{a_1(m-1)}{2mn}.
\end{equation}
Here, the constant $\mathcal{C}_M>0$ is related to the $L^1$-mass $M$ of barenblatt solution. By \cite{Va1}, there exists a constant $c^{\ast}=c^{\ast}(m,n)>0$ such that
\begin{equation}\label{eq-constant-mathcfacl-C-M-by-c-star-sigma}
\mathcal{C}_M=\left(c^{\ast}M^{a_3}\right)^{m-1} \qquad \left(a_3=\frac{2}{n}a_1\right).
\end{equation}
Denote by $\rho_{_M}(t)$ the radius of the support of Barenblatt solution $\mathcal{B}_{M}$ at time $t$, i.e.,
\begin{equation*}
x\in\textbf{supp}\,\mathcal{B}_{M}\left(\cdot,t\right)\qquad \iff \qquad |x|<\sqrt{\frac{\left(c^{\ast}M^{a_3}\right)^{m-1}}{k}}\,t^{\,a_2}=\rho_{_M}(t).
\end{equation*}
Then by an argument similar to the proof of Lemma 3.5 of \cite{KV}, we have the following lemma.
\begin{lemma}\label{lem-Lemma-3-5-of-cite-KV}
$\mathcal{B}_M\left(x,t\right)>\mathcal{B}_M\left(x,t+\tau\right)$ in a region $|x|\leq c\left(\tau,m,n\right)\rho_{_M}(t)$ and $\mathcal{B}_M\left(x,t+\tau\right)>\mathcal{B}_M\left(x,t\right)$ for $c\left(\tau,m,n\right)\rho_{_M}(t)<|x|<\rho_{_{M}}\left(t+\tau\right)$. Moreover 
\begin{equation*}
c\left(\tau,m,n\right)\to c_{\sharp}=\sqrt{\left(m-1\right)a_1}<1 \qquad \mbox{as $\tau\to 0$}. 
\end{equation*}
\end{lemma}

\subsection{Properties of solutions with Barenblatt solution $\mathcal{B}_M$ as diffusion coefficients}
For any $M\geq M_0>0$, let $w$ be a solution of 
\begin{equation}\label{eq-equation-with-barenblatt-sol-as-diffusion-coefficients}
\begin{aligned}
w_t=\nabla\cdot\left(m\,\mathcal{B}_{M}^{m-1}\nabla w\right) \qquad \forall (x,t)\in\R^n\times\left(0,\infty\right)
\end{aligned}
\end{equation} 
with initial value $w_0\in L^1\left(\R^n\right)$ which satisfies
\begin{equation}\label{eq-uniformly-upper-bound-of-u-by-B-M}
w(x,t)\leq \mathcal{B}_{M}(x,t) \qquad \forall (x,t)\in\R^n\times\left[0,\infty\right)
\end{equation}
and
\begin{equation}\label{eq-mass-conservation-of-solution-of=diffusion-coefficients-B-M}
\int_{\R^n}w(x,t)\,dx=M_0 \qquad  \forall t\geq 0.
\end{equation}

In the following lemma, we find $L^{\infty}$ bounds of solution $u$.
\begin{lemma}\label{lem-uniform-upper-bound-of-souiton-u-with-mass-M-0-by-M-2-times-something}
Let $w$ be a solution of \eqref{eq-equation-with-barenblatt-sol-as-diffusion-coefficients} and \eqref{eq-mass-conservation-of-solution-of=diffusion-coefficients-B-M}. Suppose that 
\begin{equation}\label{eq-suppose-upper-bound-of-u-by-M-1-B-1-x-t}
w(x,t)\leq \frac{M_1}{M}\mathcal{B}_{M}(x,t) \qquad \forall (x,t)\in \R^n\times(0,\infty)
\end{equation}
for any constant $M_0<M_1$. Then there exists a constant $M_2\in\left(M_0,M_1\right)$ such that  
\begin{equation}\label{eq-uniform-upper-bound-of-u-by-M-2-mathcal-B-M-at-zero}
w(x,t)\leq \frac{M_2}{M}\mathcal{B}_{M}\left(0,t\right)=c^{\ast}M_2M^{a_3-1}t^{-a_1}, \qquad \forall t>0
\end{equation}
where constants $a_1$, $a_3$ and $c^{\ast}$ are given by \eqref{eq-constant-alpha-1-beta-1-k} and \eqref{eq-constant-mathcfacl-C-M-by-c-star-sigma}.
\end{lemma}

\begin{proof}
By an argument similar to the proof of Lemma \ref{eq-same-support-between-U-and-u},
\begin{equation*}
\textbf{supp}\,w(t)=\textbf{supp}\,\mathcal{B}_M(t) \qquad \forall t>0.
\end{equation*}
We first show that $w\left(\cdot,1\right)$ does not touch $\frac{M_1}{M}\mathcal{B}_{M}(\cdot,1)$ from below at any point in $\textbf{supp}\,\mathcal{B}_M(1)$, i.e., for
\begin{equation*}
|x|<\sqrt{\frac{\left(c^{\ast}M^{a_3}\right)^{m-1}}{k}}=\rho_M(1).
\end{equation*}
Suppose that $w(x,1)$ touches $\frac{M_1}{M}\mathcal{B}_M(x,1)$ at a point $x_0$ with $|x_0|<\rho_M(1)$. By radially symmetry and continuity of $\mathcal{B}_M$, there exists a constant $\epsilon_1>0$ such that
\begin{equation*}
E_{\epsilon_1}=\left\{x\in\R^n:|x|\leq |x_0|+\epsilon_1\right\}\times\left[1-\epsilon_1^2,1\right]\subset \left\{(x,t)\in\R^n\times[0,\infty):\mathcal{B}_M(x,t)>0\right\}.
\end{equation*}
On $E_{\epsilon_1}$, there exists constant $0<c<C<\infty$ such that
\begin{equation*}
c\leq \mathcal{B}_M(x,t)\leq C \qquad \forall (x,t)\in E_{\epsilon_1}.
\end{equation*}
Thus, the equation \eqref{eq-equation-with-barenblatt-sol-as-diffusion-coefficients} is uniformly parabolic on $E_{\epsilon_1}$. Therefore the function $w-\frac{M_1}{M}\mathcal{B}_M$ is the classical solution of \eqref{eq-equation-with-barenblatt-sol-as-diffusion-coefficients} on $E_{\epsilon}$ which has its maximum at the point $(x_0,1)$ inside of $E_{\epsilon_1}$ by \eqref{eq-suppose-upper-bound-of-u-by-M-1-B-1-x-t}. By Strong Maximum Principle,
\begin{equation}\label{eq-equivalent-btween-u-and-M-1-mathcal-B-1-on-large-portion-of-support}
w(x,1)\equiv \frac{M_1}{M}\mathcal{B}_M(x,1) \qquad \forall 0\leq |x|\leq |x_0|+\epsilon_1.
\end{equation}
By maximal interval argument, \eqref{eq-equivalent-btween-u-and-M-1-mathcal-B-1-on-large-portion-of-support} can be extend to the support of $\mathcal{B}_M(1)$. Since 
\begin{equation*}
\int_{\R^n}\frac{M_1}{M}\mathcal{B}_M(x,1)\,dx=M_1\neq M_0=\int_{\R^n}w(x,1)\,dx,
\end{equation*} 
the contradiction arises and the claim follows. \\
\indent By the claim, $w(x,1)<\frac{M_1}{M}\mathcal{B}_M(x,1)\leq \frac{M_1}{M}\mathcal{B}_M(0,1)=c^{\ast}M_1M^{a_3-1}$ for all $x\in\textbf{supp}\,\mathcal{B}_M(1)$. Hence there exists a constant $M_2\in\left(M_0,M_1\right)$ such that 
\begin{equation}\label{eq-uniform-bound-of-u-at-0-1-by-c-sast-M-2}
w(x,1)\leq c^{\ast}M_2M^{a_3-1} \qquad \forall x\in\R^n.
\end{equation} 
To prove \eqref{eq-uniform-upper-bound-of-u-by-M-2-mathcal-B-M-at-zero}, we consider the rescaled function
\begin{equation*}
\widehat{w}(x,t)=T^{a_1}w\left(T^{a_2}x,Tt\right), \qquad \left(T>0\right).
\end{equation*} 
Since 
\begin{equation*}
\mathcal{B}_M(x,t)=T^{a_1}\mathcal{B}_M\left(T^{a_2}x,Tt\right),
\end{equation*}
the function $\widehat{w}$ is a solution of \eqref{eq-equation-with-barenblatt-sol-as-diffusion-coefficients} which satisfies \eqref{eq-mass-conservation-of-solution-of=diffusion-coefficients-B-M} and \eqref{eq-suppose-upper-bound-of-u-by-M-1-B-1-x-t}. Then, by an argument for \eqref{eq-uniform-bound-of-u-at-0-1-by-c-sast-M-2} we have
\begin{equation*}
w(x,T)=\frac{1}{T^{a_1}}\widehat{w}\left(\frac{x}{T^{a_2}},1\right)\leq c^{\ast}M_2M^{a_3-1}T^{-a_1} \qquad \forall x\in\R^n
\end{equation*}
and the lemma follows.
\end{proof}

By \eqref{eq-uniformly-upper-bound-of-u-by-B-M} and \eqref{eq-mass-conservation-of-solution-of=diffusion-coefficients-B-M}, there exists a constant $M_0\leq M'\leq M$ such that
\begin{equation}\label{eq-upper-bound-of-u-by-smaller-one-of-B-M}
w(x,t)\leq \frac{M'}{M}\mathcal{B}_M(x,t) \qquad \forall (x,t)\in\R^n\times\left(0,\infty\right).
\end{equation}
We now consider the infimum of these bounds
\begin{equation}\label{eq-overline-M-as-infimum-of-upper-bound-by-M-mathcal-B-1}
\overline{M}=\inf\left\{M':w(x,t)\leq \frac{M'}{M}\mathcal{B}_M(x,t)\right\}.
\end{equation}
We now are going to prove that $\overline{M}=M_0$.
\begin{thm}[Uniqueness]\label{thm-for-uniqueness-of-main-convergence}
Let $0<M_0\leq M$. Let $w$ be non-negative solution of \eqref{eq-equation-with-barenblatt-sol-as-diffusion-coefficients} which satisfies \eqref{eq-uniformly-upper-bound-of-u-by-B-M} and \eqref{eq-mass-conservation-of-solution-of=diffusion-coefficients-B-M}. Then 
\begin{equation}\label{eq-uniqueness-of-u-and-M-0-mathcal-B-1}
w=\frac{M_0}{M}\mathcal{B}_{M} \qquad \mbox{a.e. in $\R^n\times(0,\infty)$}.
\end{equation}
\end{thm}
\begin{proof}
We will use a modification of the techniques of Lemma 3.5 of \cite{KV} to prove theorem. By \eqref{eq-upper-bound-of-u-by-smaller-one-of-B-M} and \eqref{eq-overline-M-as-infimum-of-upper-bound-by-M-mathcal-B-1},
\begin{equation}\label{eq-property-overline-M-as-infimum-of-upper-bound-by-M-mathcal-B-1}
\overline{M}\geq M_0 \qquad \mbox{and} \qquad w\leq \frac{\overline{M}}{M}\mathcal{B}_M \quad \mbox{in $\R^n\times(0,\infty)$}.
\end{equation}
Suppose that $\overline{M}>M_0$. By Lemma \ref{lem-uniform-upper-bound-of-souiton-u-with-mass-M-0-by-M-2-times-something}, there exists a constant $\widetilde{M}\in\left(M_0,\overline{M}\right)$ such that
\begin{equation*}
w(x,t)\leq c^{\ast}\widetilde{M}M^{a_3-1}t^{-a_1} \qquad \forall \left(x,t\right)\in\R^n\times(0,\infty).
\end{equation*}
Let 
\begin{equation*}
W\left(x,1\right)=\min\left\{c^{\ast}\widetilde{M}M^{a_3-1},\frac{\overline{M}}{M}\mathcal{B}_M(x,1)\right\} \qquad \forall x\in\R^n
\end{equation*}
and $W$ be the solution of \eqref{eq-equation-with-barenblatt-sol-as-diffusion-coefficients} in $\R^n\times(1,\infty)$ with initial data $W(x,1)$ at time $t=1$. By maximum principle,
\begin{equation*}
w(x,t)\leq W(x,t)\leq \frac{\overline{M}}{M}\mathcal{B}_M(x,t) \qquad \forall (x,t)\in\R^n\times\left[1,\infty\right).
\end{equation*}
Since $W(0,1)=c^{\ast}\widetilde{M}M^{a_3-1}$ is strictly less than $\frac{\overline{M}}{M}\mathcal{B}_M(0,1)=c^{\ast}\overline{M}M^{a_3-1}$, by an argument similar to the proof of Lemma \ref{lem-uniform-upper-bound-of-souiton-u-with-mass-M-0-by-M-2-times-something} there exists a constant $t_1>1$ such that
\begin{equation}\label{eq-striclty-less-of-U-than-mathcal-B-1-in-support}
W\left(x,t_1\right)<\frac{\overline{M}}{M}\mathcal{B}_M\left(x,t_1\right) \qquad \forall |x|<\rho_1\left(t_1\right).
\end{equation}
By \eqref{eq-striclty-less-of-U-than-mathcal-B-1-in-support}, $W(\cdot,t_1)$ and $\frac{\overline{M}}{M}\mathcal{B}_M(\cdot,t_1)$ are strictly separated on the compact subset of $\textbf{supp}\,\mathcal{B}_M(t_1)$. Hence by Lemma \ref{lem-Lemma-3-5-of-cite-KV}, there exist constants $\delta>0$ and $\tau>0$ small enough that
\begin{equation}\label{eq-compare-U-and-mathcal-B-of-t-tau-strictly-inside-1}
W(x,t_1)<\frac{\overline{M}}{M}\mathcal{B}_M\left(x,t_1+\tau\right) \qquad \forall |x|\leq c_{\sharp}\rho_{1}\left(t_1\right)+\delta.
\end{equation}
On the other hand,
\begin{equation}\label{eq-compare-U-and-mathcal-B-of-t-tau-strictly-inside-2}
W(x,t_1)\leq \frac{\overline{M}}{M}\mathcal{B}_M(x,t_1)<\overline{M}\mathcal{B}_1\left(x,t_1+\tau\right) \qquad \forall c_{\sharp}\rho_{1}\left(t_1\right)+\delta\leq |x|\leq \rho_{1}\left(t_1+\tau\right).
\end{equation}
By \eqref{eq-compare-U-and-mathcal-B-of-t-tau-strictly-inside-1} and \eqref{eq-compare-U-and-mathcal-B-of-t-tau-strictly-inside-2},
\begin{align}
&W(x,t_1)<\frac{\overline{M}}{M}\mathcal{B}_M\left(x,t_1+\tau\right) \qquad \forall |x|\leq \rho_{1}\left(t_1+\tau\right)\notag\\
&\qquad \Rightarrow \qquad W(x,t_1)\leq\frac{\left(\overline{M}-\epsilon\right)}{M}\mathcal{B}_M\left(x,t_1+\tau\right) \qquad \forall x\in\R^n\label{eq-compare-U-and-mathcal-B-of-t-tau-times-minus-epsilon-at-t-1}
\end{align}
for sufficiently small constant $\epsilon>0$.  By \eqref{eq-compare-U-and-mathcal-B-of-t-tau-times-minus-epsilon-at-t-1} and maximum principle,
\begin{equation}\label{eq-compare-U-and-mathcal-B-of-t-tau-times-minus-epsilon}
W(x,t)\leq\frac{\left(\overline{M}-\epsilon\right)}{M}\mathcal{B}_M\left(x,t+\tau\right) \qquad \forall x\in\R^n\,\,t\geq t_1.
\end{equation}
Since $w\leq W$ for $t\geq 1$, by \eqref{eq-compare-U-and-mathcal-B-of-t-tau-times-minus-epsilon}
\begin{equation}\label{eq-compare-u-and-mathcal-B-of-t-tau-times-minus-epsilon}
w(x,t)\leq\frac{\left(\overline{M}-\epsilon\right)}{M}\mathcal{B}_M\left(x,t+\tau\right) \qquad \forall x\in\R^n\,\,t\geq t_1.
\end{equation}
We now consider the rescaled function
\begin{equation}\label{eq-rescaling-with-theta-minus-polwer}
W_{\theta}(x,t)=\frac{1}{\theta^{a_1}}W\left(\frac{x}{\theta^{a_2}},\frac{t}{\theta}\right)
\end{equation}
where constants $a_1$ and $a_2$ are given by \eqref{eq-constant-alpha-1-beta-1-k}. Then, $W_{\theta}$ is a solution of \eqref{eq-equation-with-barenblatt-sol-as-diffusion-coefficients} in $\R^n\times\left(\theta,\infty\right)$ which satisfies on the initial data
\begin{equation*}
W_{\theta}\left(x,\theta\right)=\min\left\{c^{\ast}\widetilde{M}M^{a_3-1}\theta^{-a_1},\frac{\overline{M}}{M}\mathcal{B}_M(x,\theta)\right\} \qquad \forall x\in\R^n
\end{equation*}
since $\mathcal{B}_M$ is invariant under the rescaling \eqref{eq-rescaling-with-theta-minus-polwer}. Since 
\begin{equation*}
w(x,t)\leq W_{\theta}(x,t) \qquad \forall x\in\R^n,\,\,t\geq \theta t_1,
\end{equation*}
by an argument similar to the proof of  \eqref{eq-compare-u-and-mathcal-B-of-t-tau-times-minus-epsilon},
\begin{equation}\label{eq-compare-u-and-mathcal-B-of-t-tau-times-minus-epsilon-for-upper-theta-t-1}
w(x,t)\leq\frac{\left(\overline{M}-\epsilon\right)}{M}\mathcal{B}_M\left(x,t+\theta\tau\right) \qquad \forall x\in\R^n,\,\,t\geq \theta t_1.
\end{equation}
Letting $\theta\to 0$ in \eqref{eq-compare-u-and-mathcal-B-of-t-tau-times-minus-epsilon-for-upper-theta-t-1},
\begin{equation}\label{eq-compare-u-and-mathcal-B-of-t-tau-times-minus-epsilon-for-upper-zero}
w(x,t)\leq\frac{\left(\overline{M}-\epsilon\right)}{M}\mathcal{B}_M\left(x,t\right) \qquad \forall x\in\R^n,\,\,t>0.
\end{equation}
Hence contradiction arises and $\overline{M}=M_0$. By \eqref{eq-property-overline-M-as-infimum-of-upper-bound-by-M-mathcal-B-1},
\begin{equation*}
0\leq w(x,t)\leq \frac{M_0}{M}\mathcal{B}_M(x,t) \quad \forall (x,t)\in\R^n\times(0,\infty). 
\end{equation*}
Since $w$ has $L^1$ mass $M_0$, \eqref{eq-uniqueness-of-u-and-M-0-mathcal-B-1} holds and the theorem follows.
\end{proof}

\subsection{Convergence of $U$}

Let $M$ be the $L^1$-mass of solution $U$ of \eqref{eq-PME-satisfied-by-U-total-species}. By \cite{LV} and \cite{Va1}, it is well known that there exists the uniform convergences between $U$ and Barenblatt profile $\mathcal{B}_M$.
\begin{lemma}[cf. Theorem 2.8 of \cite{LV}]\label{lem-Theorem-2-8-of-LV}
Let $U$ be the solution of \eqref{eq-PME-satisfied-by-U-total-species} with initial data $U_0$ nonnegative, integrable and compactly supported. Let $M=\int_{\R^n}U_0(x)\,dx$. Then
\begin{equation*}
\lim_{t\to\infty}\left\|U\left(\cdot,t\right)-\mathcal{B}_{M}\left(\cdot,t\right)\right\|_{L^1}=0
\end{equation*}
Convergence holds also in uniform norm in the proper scale:
\begin{equation}\label{eq-L-infty-convergence-between-U-and-Barenblatt-solution}
\lim_{t\to\infty}t^{\alpha_1}\left\|U\left(\cdot,t\right)-\mathcal{B}_{M}\left(\cdot,t\right)\right\|_{L^{\infty}}=0 \qquad \mbox{uniformly $x\in\R^n$}.
\end{equation}
\end{lemma}

\subsection{Scaling and Uniform estimates}
Let $u$, $U$ be solutions of \eqref{eq-for-each-population-u-i}, \eqref{eq-PME-satisfied-by-U-total-species} with $L^1$-mass $M_0$, $M$, respectively. Construct the families of functions
\begin{equation}\label{eq-scaling-of-u-and-U-by-lambda-0}
u_{\lambda}\left(x,t\right)=\lambda^{a_1}u\left(\lambda^{a_2}x,\lambda t\right) \qquad \mbox{and} \qquad U_{\lambda}\left(x,t\right)=\lambda^{a_1}U\left(\lambda^{a_2}x,\lambda t\right)\qquad \left(\lambda>0\right)
\end{equation}
where the exponents $a_1$ and $a_2$ are given by \eqref{eq-constant-alpha-1-beta-1-k}. Then by \eqref{eq-for-each-population-u-i} and \eqref{eq-basic-property-of-u-smaller-than-U}, $u_{\lambda}$ are solutions of 
\begin{equation}\label{eq-for-each-population-u-lambda-after-rescaling}
\begin{cases}
\begin{aligned}
\left(u_{\lambda}\right)_t&=\nabla\cdot\left(m\,U_{\lambda}^{m-1}\nabla u_{\lambda}\right)\qquad \qquad\mbox{in $\R^n\times(0,\infty)$}\\
u_{\lambda}(x,0)&=u_{0}\left(\lambda^{a_2}x\right)=u_{0,\lambda}(x)\qquad \qquad \qquad \qquad\forall x\in\R^n
\end{aligned}
\end{cases}
\end{equation}
which satisfies 
\begin{equation}\label{eq-basic-condition-between-u-lambda-and-U-lambda}
0\leq u_{\lambda}(x,t)\leq U_{\lambda}(x,t) \qquad \forall (x,t)\in\R^n\times[0,\infty).
\end{equation}
By Lemma \ref{lem-Mass-conservation-of-PME} and Lemma \ref{lem-Mass-conservation-of-u},
\begin{equation}\label{eq-mass-conservation-of-U-lambda}
\int_{\R^n}U_{\lambda}\left(x,t\right)\,dx=\int_{\R^n}\lambda^{a_1}U\left(\lambda^{a_2}x,\lambda t\right)\,dx=\int_{\R^n}U\left(y,\lambda t\right)\,dy=M<\infty\qquad \forall \lambda>0,\,\,t\geq 0
\end{equation}
and
\begin{equation}\label{eq-general-L-1-bound-of-u-lambda-at-each-time-t-except-zero}
\int_{\R^n}u_{\lambda}\left(x,t\right)\,dx=\int_{\R^n}\lambda^{a_1}u\left(\lambda^{a_2}x,\lambda t\right)\,dx=\int_{\R^n}u\left(y,\lambda t\right)\,dy=M_0<\infty\qquad \forall \lambda>0,\,\,t\geq 0.
\end{equation}
Hence the family $\left\{u_{\lambda}\right\}_{\lambda\geq 1}$ is uniformly bounded in $L^1\left(\R^n\right)$ for all $t>0$. By \eqref{eq-L-infty-bound-of-diffusion-coefficients-U-by-L-infty-of-U-0-with-positive-t} and \eqref{eq-basic-property-of-u-smaller-than-U},
\begin{equation*}
\left\|u_{\lambda}\left(\cdot,1\right)\right\|_{L^{\infty}}\leq \left\|U_{\lambda}\left(\cdot,1\right)\right\|_{L^{\infty}}=\lambda^{a_1}\left\|U\left(\cdot,\lambda\right)\right\|_{L^{\infty}}\leq\lambda^{a_1}\frac{C\left\|U_0\right\|_{L^1}^{\frac{2a_1}{n}}}{\lambda^{a_1}}=CM^{\frac{2a_1}{n}}
\end{equation*}
which is independent to $\lambda$. Similarly,
\begin{equation}\label{eq-general-L-p-bound-of-u-lambda-at-each-time-t-except-zero}
\left\|u_{\lambda}\left(\cdot,t_0\right)\right\|_{L^{\infty}}\leq CM^{\frac{2a_1}{n}}t_0^{-a_1} \qquad \forall t_0>0.
\end{equation}
By \eqref{eq-general-L-1-bound-of-u-lambda-at-each-time-t-except-zero}, \eqref{eq-general-L-p-bound-of-u-lambda-at-each-time-t-except-zero} and Interpolation theory, 
\begin{equation}\label{eq-equi-bound-of-u-lambda-in-L-p}
\mbox{$\left\|u_{\lambda}\left(\cdot,t\right)\right\|_{L^p}$ is equibounded for all $p\in\left[1,\infty\right]$.}
\end{equation}
By Lemma \ref{lem-Theorem-2-8-of-LV} and \eqref{eq-mass-conservation-of-U-lambda}, there exists a constant $\lambda_0>0$ such that for any $\lambda\geq\lambda_0$ there exist constants $0<c_{\lambda}$, $t_{\lambda}<1$ such that
\begin{equation}\label{comparison-between-mathcal-B-M-at-t-lambda-and-U-lambda-at-0}
c_{\lambda}\mathcal{B}_M\left(x,t_{\lambda}\right)\leq U_{\lambda}(x,0) \qquad \forall x\in\R^n,\,\,\lambda\geq \lambda_0.
\end{equation}
Here, 
\begin{equation}\label{eq-limit-of-c-lambda-to-1-and-t-lambda-to-0-as-lambda-to-infty}
c_{\lambda}\to 1 \qquad \mbox{and} \qquad t_{\lambda}\to 0 \qquad \mbox{as $\lambda\to\infty$}. 
\end{equation}
By \eqref{comparison-between-mathcal-B-M-at-t-lambda-and-U-lambda-at-0} and the maximum principle for porous medium equation, \cite{Va1}, we have
\begin{align}
&c_{\lambda}\mathcal{B}_M\left(x,t+t_{\lambda}\right)\leq U_{\lambda}(x,t) \qquad \forall x\in\R^n,\,\,t>0,\,\,\lambda\geq \lambda_0\notag\\
&\qquad \Rightarrow \qquad  c_{\lambda}\mathcal{B}_M\left(x,t_0+t_{\lambda}\right)\leq U_{\lambda}(x,t_0) \qquad \forall x\in\R^n,\,\,\lambda\geq \lambda_0\label{comparison-between-mathcal-B-M-at-t-0+t-lambda-and-U-lambda-at-t-0}
\end{align}
for any $t_0>0$. Observe  that
\begin{equation}\label{eq-increasing-of-support-of-barrenblatt-B-M-as-t-increases}
\textbf{supp}\,\mathcal{B}_{M}\left(x,t_0\right)\subset \textbf{supp}\,\mathcal{B}_{M}\left(x,t_0+t_{\lambda}\right) \qquad \forall \lambda\geq\lambda_0.
\end{equation}
Since $\mathcal{B}_M$ is continuous in $\R^n\times(0,\infty)$, by \eqref{eq-limit-of-c-lambda-to-1-and-t-lambda-to-0-as-lambda-to-infty} and \eqref{eq-increasing-of-support-of-barrenblatt-B-M-as-t-increases} there exists a constant $\lambda_1(t_0)>\lambda_0$ such that
\begin{equation}\label{eq-small-difference-between-barenblatt-solutions-when-times-difference-is-small}
c_{\lambda}\geq \frac{3}{4}\qquad \mbox{and}\qquad \frac{2}{3}\mathcal{B}_M\left(x,t_0\right)\leq \mathcal{B}_M\left(x,t_0+t_{\lambda}\right) \qquad \forall \lambda\geq\lambda_1.
\end{equation}
By \eqref{comparison-between-mathcal-B-M-at-t-0+t-lambda-and-U-lambda-at-t-0} and \eqref{eq-small-difference-between-barenblatt-solutions-when-times-difference-is-small},
\begin{equation}\label{eq-compare-between-half-of-barrenblatt-solution-and-U-lambda-at-t-0}
\frac{1}{2}\mathcal{B}_{M}(x,t_0)\leq U_{\lambda}(x,t_0) \qquad \forall\lambda\geq\lambda_1.
\end{equation}
By \eqref{eq-compare-between-half-of-barrenblatt-solution-and-U-lambda-at-t-0} and the maximum principle for porous medium equation, \cite{Va1}, we have
\begin{equation}\label{eq-compare-between-half-of-barrenblatt-solution-and-U-lambda-after-t-0}
\frac{1}{2}\mathcal{B}_{M}(x,t)\leq U_{\lambda}(x,t) \qquad \forall t\geq t_0,\,\,\lambda\geq\lambda_1.
\end{equation}
Multiplying the first equation in \eqref{eq-for-each-population-u-lambda-after-rescaling} by $u_{\lambda}$ and integrating over $\R^n\times(t_0,t)$ for all $t>t_0$, the we have
\begin{align}
&\int_{\R^n}u_{\lambda}^2(x,t)\,dx+m\int_{t_0}^{t}\int_{\R^n}U_{\lambda}^{m-1}\left|\nabla u_{\lambda}\right|^2\,dxd\tau= \int_{\R^n}u_{\lambda}^2(x,t_0)\,dx\notag\\
& \Rightarrow \qquad \int_{t_0}^{t}\int_{\R^n}U_{\lambda}^{m-1}\left|\nabla u_{\lambda}\right|^2\,dxd\tau\leq C\left(\left\|u_{\lambda}(t_0)\right\|_{L^2}\right) \qquad \forall t\geq t_0>0\notag\\
& \Rightarrow \qquad \int_{t_0}^{t}\int_{\R^n}\mathcal{B}_M^{m-1}\left|\nabla u_{\lambda}\right|^2\,dxd\tau\leq C\left(\left\|u_{\lambda}(t_0)\right\|_{L^2}\right) \qquad \forall t\geq t_0>0,\,\,\lambda\geq\lambda_1.\label{eq-equi-bound-of-u-lambda-in-weight-H-1-with-weight-Barrenblatt-M}
\end{align}

\subsection{Limit function of solution $u$}

As the first result of the convergence, we prove that there exists an uniform convergence  between $u$ and $\frac{M_0}{M}\mathcal{B}_M$ in $L^p$.
 
\begin{proof}[\textbf{Proof of Theorem \ref{eq-convergence-in-L-1-and-L-infty-of-u-to-some-constant-Barenblatt}}]
We will use a modification of the proof of Theorem 18.1 of \cite{Va1}. For any $\lambda>0$, let $u_{\lambda}$, $U_{\lambda}$ be given by \eqref{eq-scaling-of-u-and-U-by-lambda-0}. By \eqref{eq-general-L-p-bound-of-u-lambda-at-each-time-t-except-zero}, the family $\left\{u_{\lambda}\right\}_{\lambda>0}$ is uniformly bounded in $\R^n\times\left(t_0,\infty\right)$ for any $t_0>0$. Thus $\left\{u_{\lambda}\right\}_{\lambda>0}$ is relatively compact in $L^1_{loc}\left(\R^n\times(0,\infty)\right)$. Therefore for sequence $\lambda_n\to \infty$ as $n\to\infty$, the sequence $\left\{u_{\lambda_n}\right\}$ has a subsequence which we may assume without loss of generality to be the sequence itself that converges in $L^1_{loc}\left(\R^n\times(0,\infty)\right)$ to some function $u_{\infty}$ in $\R^n\times(0,\infty)$ as $n\to\infty$.\\ 
\indent Let $0<t_0<t_1$ and let $\vp\in C^{\infty}_0\left(\,\R^n\times(0,\infty)\right)$ be a test function such that 
\begin{equation*}
\vp(\cdot,t)=0 \qquad \forall 0<t<t_0,\,\,t>t_1.
\end{equation*}
Multiplying the first equation in \eqref{eq-for-each-population-u-lambda-after-rescaling} by $\vp\in C^{\infty}_0\left(\,\R^n\times(0,\infty)\right)$ and integrating over $\R^n\times(0,\infty)$, we have
\begin{align}
&m\int_0^{\infty}\int_{\R^n}U_{\lambda}^{m-1}\nabla u_{\lambda}\cdot\nabla\vp\,dxdt-\int_0^{\infty}\int_{\R^n}u_{\lambda}\vp_t\,dxdt=0\notag\\
&\qquad \Rightarrow \qquad m\int_0^{\infty}\int_{\R^n}\mathcal{B}_{M}^{m-1}\nabla u_{\lambda}\cdot\nabla\vp\,dxdt\notag\\
&\qquad \qquad \qquad \qquad +m\int_0^{\infty}\int_{\R^n}\left(U_{\lambda}^{m-1}-\mathcal{B}_M\right)\nabla u_{\lambda}\cdot\nabla\vp\,dxdt -\int_0^{\infty}\int_{\R^n}u_{\lambda}\vp_t\,dxdt=0. \label{eq-energy-type-eq-for-u-lambda-as-test-function-u-lambda-p-and-vp-before-lambda-to-infty}
\end{align}
Let $\lambda_1>0$ be given by \eqref{eq-small-difference-between-barenblatt-solutions-when-times-difference-is-small} and let $\epsilon>0$.  Then by \eqref{eq-compare-between-half-of-barrenblatt-solution-and-U-lambda-after-t-0},
\begin{equation}\label{eq-subset-mathcal-B-M-of-U-lambda-after-t-0}
\textbf{supp}\,\mathcal{B}_M(\cdot,t) \subset\textbf{supp}U_{\lambda}(\cdot,t) \qquad \forall t\geq t_0,\,\,\lambda\geq\lambda_1.
\end{equation}  
By Lemma \ref{lem-Theorem-2-8-of-LV}, there exists a constant $\lambda_2\geq\lambda_1$ such that
\begin{equation}\label{measure-of-small-set-supp-U-lambda-bs-supp-mathcal-B-M}
\left|\textbf{supp}U_{\lambda}(t)\,\bs \textbf{supp}\,\mathcal{B}_M(t)\,\right|<\epsilon \qquad \forall t\in\left[t_0,t_1\right]\,\,\lambda\geq\lambda_2
\end{equation}
and 
\begin{equation}\label{measure-of-small-set-supp-U-lambda-bs-supp-mathcal-B-M-1}
\left|U^{m-1}_{\lambda}(x,t)-\mathcal{B}^{m-1}_M(x,t)\,\right|<\epsilon \qquad \forall x\in\R^n,\,\, t\in\left[t_0,t_1\right]\,\,\lambda\geq\lambda_2.
\end{equation}
Let
\begin{equation*}
E_{\mathcal{B}_M,t_0,t_1}=\left\{(x,t)\in\R^n\times\left[t_0,t_1\right]:\mathcal{B}_M(x,t)>0\right\}
\end{equation*}
and let $\mathcal{K}$ be a compact subset of $E_{\mathcal{B}_M,t_0,t_1}$ such that
\begin{equation}\label{measure-of-small-set-E-mathcal-B-M-t-0-t-1-bs-mathcal-K-compact-subset}
\left|E_{\mathcal{B}_M,t_0,t_1}\bs\mathcal{K}\right|<\epsilon.
\end{equation}
By Lemma \ref{lem-Theorem-2-8-of-LV}, there exists a constant $\lambda_3>\lambda_2$ such that  $\left\{U_{\lambda}\right\}_{\lambda\geq\lambda_3}$ is uniformly parabolic in $\mathcal{K}$. Then by standard Schauder estimates for parabolic partial differential equation, \cite{LSU}, there exists a constant $C_{\mathcal{K}}<\infty$ such that
\begin{equation}\label{eq-uniform-bound=of-gradient-of-u-lambda-in-mathcal-K}
\left|\nabla u_{\lambda}\right|\leq C_{\mathcal{K}} \qquad \forall \lambda\geq\lambda_3,\,\,(x,t)\in\mathcal{K}.
\end{equation} 
By \eqref{eq-equi-bound-of-u-lambda-in-weight-H-1-with-weight-Barrenblatt-M}, \eqref{eq-subset-mathcal-B-M-of-U-lambda-after-t-0}, \eqref{measure-of-small-set-supp-U-lambda-bs-supp-mathcal-B-M}, \eqref{measure-of-small-set-supp-U-lambda-bs-supp-mathcal-B-M-1}, \eqref{measure-of-small-set-E-mathcal-B-M-t-0-t-1-bs-mathcal-K-compact-subset} and \eqref{eq-uniform-bound=of-gradient-of-u-lambda-in-mathcal-K},
\begin{equation*}
\begin{aligned}
&\left|\int_0^{\infty}\int_{\R^n}\left(U_{\lambda}^{m-1}-\mathcal{B}^{m-1}_M\right)\nabla u_{\lambda}\cdot\nabla\vp\,dxdt\right|\\
&\qquad \qquad \leq \int_0^{\infty}\int_{\R^n}\left|U_{\lambda}^{m-1}-\mathcal{B}^{m-1}_M\right|\nabla u_{\lambda}\cdot\nabla\vp\,dxdt\\
&\qquad \qquad \leq \iint_{\mathcal{K}}\left|U_{\lambda}^{m-1}-\mathcal{B}^{m-1}_M\right|\nabla u_{\lambda}\cdot\nabla\vp\,dxdt+\iint_{E_{\mathcal{B}_M,t_0,t_1}\bs\mathcal{K}}\left(U_{\lambda}^{m-1}+\mathcal{B}^{m-1}_M\right)\nabla u_{\lambda}\cdot\nabla\vp\,dxdt\\
&\qquad \qquad \qquad \qquad +\int_{t_0}^{t_1}\int_{\textbf{supp}\,U_{\lambda}(t)\bs\textbf{supp}\,\mathcal{B}^{m-1}_M(t)}U_{\lambda}^{m-1}\nabla u_{\lambda}\cdot\nabla\vp\,dxdt\\
&\qquad \qquad \leq C_{\mathcal{K}}\left|\mathcal{K}\right|\left\|\nabla\vp\right\|_{L^{\infty}}\epsilon+C\left(\left\|u_{\lambda}(t_0)\right\|_{L^2}\right)\left(\left(1+\sqrt{t_1-t_0}\right)\left\|U_{\lambda}\right\|^{\frac{m-1}{2}}_{L^{\infty}}+\left\|\mathcal{B}_{M}\right\|^{\frac{m-1}{2}}_{L^{\infty}}\right)\left\|\nabla\vp\right\|_{L^{\infty}}\epsilon^{\frac{1}{2}} \qquad \forall \lambda\geq\lambda_3.
\end{aligned}
\end{equation*}
Since $\epsilon$ is arbitrary,
\begin{equation}\label{eq-convergence-of-extra-term-with-U-lambda-minus-Barenbaltt-B-M}
\int_0^{\infty}\int_{\R^n}\left(U_{\lambda}^{m-1}-\mathcal{B}^{m-1}_M\right)\nabla u_{\lambda}\cdot\nabla\vp\,dxdt\to 0 \qquad \mbox{as $\lambda\to0$}.
\end{equation}
By \eqref{eq-equi-bound-of-u-lambda-in-weight-H-1-with-weight-Barrenblatt-M},
\begin{equation}\label{eq-uniform-locally-convergence-of-in-L-p-infty-H-1}
\begin{cases}
\begin{aligned}
u_{\lambda}&\to u_{\infty} \qquad \qquad  \mbox{locally in $L^{1}$}\\
\nabla u_{\lambda}&\to\nabla u_{\infty} \qquad \mbox{locally in $L^2$ with weight $\mathcal{B}_M^{m-1}$}
\end{aligned}
\end{cases}
\end{equation}
Letting $\lambda\to\infty$ in \eqref{eq-energy-type-eq-for-u-lambda-as-test-function-u-lambda-p-and-vp-before-lambda-to-infty}, by \eqref{eq-convergence-of-extra-term-with-U-lambda-minus-Barenbaltt-B-M} and \eqref{eq-uniform-locally-convergence-of-in-L-p-infty-H-1} we have
\begin{equation*}\label{eq-energy-type-eq-for-u-lambda-as-test-function-u-lambda-p-and-vp-after-lambda-to-infty}
\begin{aligned}
m\int_0^{\infty}\int_{\R^n}\mathcal{B}_M^{m-1}\nabla u_{\infty}\cdot\nabla\vp\,dxdt-\int_0^{\infty}\int_{\R^n}u_{\infty}\vp_t\,dxdt=0.
\end{aligned}
\end{equation*}
Thus   
\begin{equation}\label{eq-u-infty-weak-soluiton-of-equation-removing-u-to-p}
\mbox{$u_{\infty}$ is a weak solution of $u_t-\nabla\left(m\,\mathcal{B}_M^{m-1}\nabla u\right)=0$} \qquad \mbox{in $\R^n\times(0,\infty)$}.
\end{equation}
By an argument similar to the proofs of Lemma 18.4 and Lemma 18.6 of \cite{Va1}, we can also have
\begin{equation}\label{eq-initial-conditions-of-u-lambda-as-lambda0to-infty-and-u-infty-x-t-as-t-to-zoer}
u_{0,\lambda}(x)\to M_0\delta(x) \quad \mbox{as $\lambda\to\infty$}  \qquad \mbox{and} \qquad u_{\infty}(x,t)\to M_0\delta(x) \qquad \mbox{as $t \to 0$}
\end{equation}
By \eqref{eq-basic-condition-between-u-lambda-and-U-lambda}, \eqref{eq-u-infty-weak-soluiton-of-equation-removing-u-to-p} and \eqref{eq-initial-conditions-of-u-lambda-as-lambda0to-infty-and-u-infty-x-t-as-t-to-zoer}, $u_{\infty}$ is a solution of \eqref{eq-equation-with-barenblatt-sol-as-diffusion-coefficients} which satisfies \eqref{eq-uniformly-upper-bound-of-u-by-B-M} and \eqref{eq-mass-conservation-of-solution-of=diffusion-coefficients-B-M}. Thus by \eqref{eq-general-L-1-bound-of-u-lambda-at-each-time-t-except-zero} and Theorem \ref{thm-for-uniqueness-of-main-convergence},
\begin{equation}\label{eq-equality-of-u-infty-to-M-0-over-M-times-B-M}
u_{\infty}(x,t)=\frac{M_0}{M}\mathcal{B}_M(x,t) \qquad \forall (x,t)\in\R^n\times\left[0,\infty\right).
\end{equation}
By \eqref{eq-equality-of-u-infty-to-M-0-over-M-times-B-M} and an argument similar to the the proof of Theorem 2.8 of \cite{LV}, we have \eqref{eq-convergence-in-L-1-from-u-to-barrenblatt-M-0-over-M-B-M}, \eqref{eq-convergence-in-L-infty-from-u-to-barrenblatt-M-0-over-M-B-M} and the theorem follows.
\end{proof}

\subsection{$C_s^{\infty}$-convergence}
We finish this section by improving Theorem \ref{eq-convergence-in-L-1-and-L-infty-of-u-to-some-constant-Barenblatt} (the uniform convergence in $L^p$, $p\geq 1$) up to $C_s^{\infty}$-convergence. Suppose that $U_0$ and $u_0$ satisfy the \textbf{Conditions for $C^{\infty}_s$-convergence} given in Introduction. Then, by \eqref{eq-non-degeneracy-of-pressure-of-sol-u} we can choose a sufficiently small constant $\epsilon_1>0$ such that
\begin{equation*}
\epsilon_1U_0(x)\leq u_0 \qquad \forall x\in\R^n.
\end{equation*}
Then, by \eqref{eq-basic-property-of-u-smaller-than-U} and the maximum principle for porous medium equation we can get the following lemma.
\begin{lemma}\label{lem-diffusion-coefficient-U-trapped-in-between-u-and-1-over-epsilon-u}
Under the \textbf{Conditions for $C_s^{\infty}$-convergence}, there exists a constant $\epsilon_1>0$ such that
\begin{equation*}
0\leq u(x,t)\leq U(x,t)\leq \frac{1}{\epsilon_1}u(x,t) \qquad \forall (x,t)\in\R^n\times[0,\infty).
\end{equation*}
\end{lemma}
Denote by $v$ and $V$ the pressures of $u$ and $U$ respectively, i.e.,
\begin{equation*}
v(x,t)=mu^{m-1}(x,t) \qquad \mbox{and} \qquad V(x,t)=mU^{m-1}(x,t) \qquad \forall (x,t)\in\R^n\times\left[0,\infty\right). 
\end{equation*}
Then 
\begin{align}
u_t&=\nabla\left(A\,\nabla u^m\right) \qquad \mbox{in $\R^n\times(0,\infty)$}\notag\\
& \Rightarrow \qquad v_t=A\left(v\,\La v+\frac{1}{m-1}\left|\nabla v\right|^2\right)+v\,\nabla A\cdot\nabla v \qquad \mbox{in $\R^n\times(0,\infty)$}\label{eq-equation-for-pressure-v-3}
\end{align}
where $A=\left(\frac{U}{u}\right)^{m-1}$.\\
\indent To explain the concept of $C_s^{\infty}$-convergence, we first review the metric $ds$ and the space $C^{\infty}_s$ introduced in \cite{DH}. Consider the change of coordinates by which the free boundary $v=0$ has been transformed into the fixed boundary. By Implicit Function Theorem, we can solve the equation $z=v(x_1,\cdots,x_{n-1},x_n,t)$ with respect to $x_n$ locally around the points $\left(x^0_1,\cdots,x^0_{n-1},x^0_n,t^0\right)$ on free boundary, i.e., for sufficiently small $\eta>0$ there exists a function $x_n=h(x_1,\cdots,x_{n-1},z,t)$ defined on a small box 
\begin{equation*}
\mathcal{B}_{\eta}=\left\{0\leq z\leq \eta,\left|x_i-x^0_i\right|\leq\eta,-\eta\leq t-t^0\leq 0\right\} \qquad \forall i=1,\cdots,n-1.
\end{equation*} 
On the set $\mathcal{B}_{\eta}$,
\begin{equation}\label{eq-function-h-from-v-by-Implicit-Function-Theorem} 
z=v\left(x',h\left(x',z,t\right),t\right) \qquad \qquad \big(x'=\left(x_1,\cdots,x_{n-1}\right)\big).
\end{equation}
Thus by simple computation, we have
\begin{equation}\label{eq-relation-of-v-and-h-differentiables-etc}
\begin{aligned}
v_{x_n}&=\frac{1}{h_z}, \quad v_{x_i}=-\frac{h_{x'}}{h_z}, \quad v_t=-\frac{h_{t}}{h_z}\\
v_{x_nx_n}=-\frac{h_{zz}}{h_z^3},&\quad v_{x_ix_i}=-\frac{1}{h_z}\left(\frac{h_{x_i}^2}{h_z^2}h_{zz}-\frac{2h_{x_i}}{h_z}h_{x_iz}+h_{x_ix_i}\right) \qquad \forall i=1,\cdots,n-1.
\end{aligned}
\end{equation}
Then, by \eqref{eq-equation-for-pressure-v-3} and \eqref{eq-relation-of-v-and-h-differentiables-etc} $h$ satisfies
\begin{align}
 h_t&=A\,z\La_{x'}h+A\,z^{-\sigma}\left(z^{1+\sigma}F\left(\nabla h\right)\right)_z+z\nabla_{x'}A\cdot\nabla_{x'}h+zA_z\,F\left(\nabla h\right)\notag\\
 &=z^{-\sigma}\nabla_{x'}\left(Az^{1+\sigma}\nabla_{x'}h\right)+\,z^{-\sigma}\left(Az^{1+\sigma}F\left(\nabla h\right)\right)_z\label{eq-equ-for-h-transformed-into-fixed-boundary-z-equal-0}
\end{align}
where
\begin{equation*}
\sigma=\frac{1}{m-1}-1 \qquad \mbox{and} \qquad F\left(\nabla h\right)=-\frac{1+\left|\nabla_{x'}h\right|^2}{h_z}.
\end{equation*}
By \eqref{eq-basic-property-of-u-smaller-than-U} and Lemma \ref{lem-diffusion-coefficient-U-trapped-in-between-u-and-1-over-epsilon-u}, $A$ is uniformly parabolic in $\R^n\times(0,\infty)$. Therefore, by an argument similar to the paper \cite{DH} it can be easily checked that the equation \eqref{eq-equ-for-h-transformed-into-fixed-boundary-z-equal-0} is governed by the Riemannian metric $ds$ where
\begin{equation*}
ds^2=\frac{dx_1^2+\cdots+dx_{n-1}^2+dz^2}{2z}.
\end{equation*}
The distance between two points $P_1=\left(x_1^1,\cdots,x_{n-1}^1,z^1,t^1\right)$ and $P_2=\left(x_1^2,\cdots,x_{n-1}^2,z^2,t^2\right)$ in this metric is equivalent to the function
\begin{equation*}
\overline{s}\left[P_1,P_2\right]=\frac{\sum_{i=1}^{n-1}\left|x_i^1-x_i^2\right|+\left|z^1-z^2\right|}{\sum_{i=1}^{n-1}\sqrt{x_i}+\sqrt{\left|z^1-z^2\right|}}+\sqrt{\left|t^1-t^2\right|}.
\end{equation*}
Under this distance, H\"older semi-norm, $C_s^{\alpha}$ norm and $C_s^{2+\alpha}$ norm of a function $f$ defined on a compact subset $\mathcal{A}$ of the half space $\left\{(x_1,\cdots,x_{n-1},z,t):z\geq 0\right\}$ are given as follow.
\begin{equation*}
\begin{aligned}
\left\|f\right\|_{H^{\alpha}_s\left(\mathcal{A}\right)}&=\sup\left\{\frac{\left|f(P_1)-f(P_2)\right|}{s\left[P_1,P_2\right]^{\alpha}}\,:\,\forall P_1,\,P_2\in\mathcal{A}\right\}\\
\left\|f\right\|_{C^{\alpha}_s\left(\mathcal{A}\right)}&=\left\|f\right\|_{L^{\infty}\left(\mathcal{A}\right)}+\left\|f\right\|_{H^{\alpha}_s\left(\mathcal{A}\right)}\\
\left\|f\right\|_{C^{2+\alpha}_s\left(\mathcal{A}\right)}&=\left\|f\right\|_{C^{\alpha}_s\left(\mathcal{A}\right)}+\sum_{i=1}^{n-1}\left\|f_{x_i}\right\|_{C^{\alpha}_s\left(\mathcal{A}\right)}+\left\|f_z\right\|_{C^{\alpha}_s\left(\mathcal{A}\right)}+\left\|f_t\right\|_{C^{\alpha}_s\left(\mathcal{A}\right)}\\
&\qquad +\sum_{i,j=1}^{n-1}\left\|zf_{x_ix_j}\right\|_{C^{\alpha}_s\left(\mathcal{A}\right)}+\sum_{i=1}^{n-1}\left\|zf_{x_iz}\right\|_{C^{\alpha}_s\left(\mathcal{A}\right)}+\left\|zf_{zz}\right\|_{C^{\alpha}_s\left(\mathcal{A}\right)}.
\end{aligned}
\end{equation*}
The concept of $C_s^{\infty}$ space can be obtained by extending these definitions to spaces of higher order derivatives. For any $k\in\N$, we denote by $C_s^{k,\epsilon_1+\alpha}\left(\mathcal{A}\right)$, $\left(\epsilon_1=0,\,2\right)$, the space of all functions $f$ whose $k$-th  order derivatives $D_{x_1}^{i_1}\cdots D_{x_{n-1}}^{i_{n-1}}D_z^jD_t^lf$, $\left(i_1+\cdots+i_{n-1}+j+l=k\right)$, exists and belong to the space of $C_{s}^{\epsilon_1+\alpha}\left(\mathcal{A}\right)$. Then we say that a function $f$ belongs to the space $C_s^{\infty}\left(\mathcal{A}\right)$ by
\begin{equation*}
f\in C_s^{\infty}\left(\mathcal{A}\right)\qquad \iff \qquad f\in C_s^{k,2+\alpha}\left(\mathcal{A}\right) \qquad \forall k\in\N.
\end{equation*}
\indent From now on, we are going to focus on $C_s^{\infty}$-convergence. For any $\lambda>0$, let $v_{\lambda}$ be the rescaled function of $v$ by
\begin{equation*}
v_{\lambda}(x,t)=\lambda^{(m-1)a_1}v\left(\lambda^{a_2}x,\lambda t\right), \qquad \forall \lambda>0,\,\,(x,t)\in\R^n\times(0,\infty)
\end{equation*}
where the exponents $a_1$ and $a_2$ are given by \eqref{eq-constant-alpha-1-beta-1-k}. Let $h_{\lambda}$ be the function from \eqref{eq-function-h-from-v-by-Implicit-Function-Theorem}  with $v$ being replaced by $v_{\lambda}$. By Theorem 4.3 of \cite{LV}, there exists a constant $\lambda_0>0$ such that
\begin{equation}\label{eq-expression-of-C-1-alpha-smooth-of-free-boundary-of-v-lambda}
\mbox{The free boundary $\partial\left\{(x,t):v_{\lambda}(x,t)>0\right\}$ is $C^{1,\alpha}$ surface for all $\lambda>\lambda_0$}.
\end{equation}
By \eqref{eq-basic-property-of-u-smaller-than-U} and Lemma \ref{lem-diffusion-coefficient-U-trapped-in-between-u-and-1-over-epsilon-u}, the coefficients $A(x,t)$ is uniformly parabolic in $\R^n\times(0,\infty)$. Moreover, by an argument similar to the proof of Theorem \ref{cor-local-holder-estimates-of-u-i-s-in-solution-bold-u} we have
\begin{equation}\label{eq-C-s-alpha-continuity-of-solution-v-lambda-and-h-lambda}
A\in C_s^{\,\alpha}.
\end{equation}
Thus the equation \eqref{eq-equation-for-pressure-v-3} belongs to the same class of equations studied in \cite{Ko}. Hence by \eqref{eq-C-s-alpha-continuity-of-solution-v-lambda-and-h-lambda} and an argument similar to the proof of Theorem 5.6.1 in \cite{Ko}, $h_{\lambda}$ have the $C_s^{1,\alpha}$-estimates up to the boundary. Applying the standard bootstrap argument, we can even get $C_s^{\,k,\alpha}$-estimates of $h_{\lambda}$ for any $k\in\N$. Therefore, we can get the uniform estimate of derivatives of $v_{\lambda}$ with respect to the original $(x,t)$. 
\begin{thm}[cf. Theorem 3.1 of \cite{LV}]\label{thm-theorem-3-1-of-cite-LV}
For every $k\in\N$, there exist constants $\lambda_k>0$ and $C_k>0$ such that
\begin{equation*}
\left\|v_{\lambda}\right\|_{C_{s}^{k}\left(\overline{\Omega_0\left(u_{\lambda}\right)}\right)}<C_k \qquad \forall \lambda>\lambda_k
\end{equation*}
where
\begin{equation*}
\Omega_0\left(v_{\lambda}\right)=\left\{(x,t):v_{\lambda}(x,t)>0,\,\,1<t<2\right\}.
\end{equation*}
\end{thm}
We finish this work by proving the Theorem \ref{thm-Theorem-3-2-of-cite-LV}.
\begin{proof}[\textbf{Proof of Theorem \ref{thm-Theorem-3-2-of-cite-LV}}]
Since the proof is almost same as that of Theorem 3.2 of \cite{LV}, we will give a sketch of proof here. By Theorem \ref{eq-convergence-in-L-1-and-L-infty-of-u-to-some-constant-Barenblatt}, there is the uniform convergence such that
\begin{equation*}
v_{\lambda}(x,t)\to \left(\frac{M_0}{M}\mathcal{B}_M(x,t)\right)^{m-1}:=t^{-a_1(m-1)}G\left(\frac{x}{t^{a_2}}\right)  \qquad \mbox{as $\lambda\to\infty$}
\end{equation*}
where $a_1$ and $a_2$ are given by \eqref{eq-constant-alpha-1-beta-1-k}. This immediately gives us a parametrization of $v_{\lambda}$, i.e., there exists a function $g_{\lambda}(x,t)$ such that
\begin{equation*}
\left(x,v_{\lambda}(x,t)\right)=\left(x,t^{-a_1(m-1)}G\left(\frac{x}{t^{a_2}}\right)\right)+g_{\lambda}(x,t)N\left(\frac{x}{t^{a_2}}\right)
\end{equation*}
where $N(x)$ is a smooth unit vector field, transverse to the surface $\left(x,G(x)\right)$ and parallel to the $x$-plane in a neighborhood of the boundary $\partial\left\{x:G(x)>0\right\}$. By Theorem \ref{thm-theorem-3-1-of-cite-LV} and Arzel\'a-Ascoli Theorem, there exists the $C_s^{\infty}$-convergence between $v_{\lambda}$ and $ \left(\frac{M_0}{M}\mathcal{B}_M(x,t)\right)^{m-1}$ and the theorem follows.
\end{proof}
\noindent {\bf Acknowledgement:} Sunghoon Kim was supported by the Research Fund, 2019 of The Catholic University of Korea.
K. Lee has been supported by  grant funded by    Samsung Science and Technology Foundation under Project Number SSTF-BA170110840  Samsung Science \& Technology Foundation (SSTF) under Project Number SSTF-BA1701-03 .
K. Lee also holds a joint appointment with the Research Institute of Mathematics of Seoul National University.

\end{document}